\documentclass{article}%

\usepackage{algorithm2e}
\usepackage{algorithmic}
\usepackage{array}
\usepackage{amsmath}%
\usepackage{amsfonts}%
\usepackage{amssymb}%
\usepackage{graphicx}
\usepackage{fullpage}
\usepackage{palatino}
\usepackage[english]{babel}
\usepackage{tikz}
\usepackage{subcaption}
\usepackage[margin=2cm]{caption}
\usepackage[bookmarksnumbered=true]{hyperref}
\usepackage[authoryear]{natbib}

\usepackage{tikz}
\usepackage{pgfplots}
\pgfplotsset{compat=newest}
\usepackage{xcolor}

\newtheorem{theorem}{Theorem}[section]

\newtheorem{assumption}[theorem]{Assumption}

\newtheorem{corollary}[theorem]{Corollary}

\newtheorem{example}[theorem]{Example}

\newtheorem{lemma}[theorem]{Lemma}

\newtheorem{proposition}[theorem]{Proposition}
\newtheorem{remark}[theorem]{Remark}

\newenvironment{proof}[1][Proof]{\textbf{#1.} }{\ \rule{0.5em}{0.5em}}

\usepackage{amsmath}%
\usepackage{amsfonts}%
\usepackage{amssymb}%
\usepackage{graphicx}
\usepackage{booktabs}
\usepackage{bbm}
\usepackage[english]{babel}
\usepackage{tikz}
\usepackage{verbatim}
\usetikzlibrary{plotmarks}
\tikzset{
    scale plot marks/.is choice,
    scale plot marks/false/.code={
        \def\pgfuseplotmark##1{\pgftransformresetnontranslations\csname pgf@plot@mark@##1\endcsname}
    },
    scale plot marks/true/.style={},
    scale plot marks/.default=true
}

\usepackage[bookmarksnumbered=true]{hyperref}

\numberwithin{equation}{section}

\title{After-sales services during an asset's lifetime: collaborative planning of system upgrades}
\author{Fiona~Sloothaak \and Alp Ak\c{c}ay \and Matthieu van der Heijden \and Geert-Jan van Houtum}

\date{\vspace{-1ex}}

\begin{document}
\maketitle

\begin{abstract}
We consider a physical asset consisting of complex systems, where the systems may require upgrades during the lifetime of the asset.  In practice, the asset owner and system supplier can make the upgrade decisions together, requiring a decision support model that can be jointly used to optimize the total benefit for both parties. Motivated by a real-life use case including an asset owner and a system supplier, we build a continuous-time model to optimize the upgrade decisions of a system during the fixed lifetime of the asset. In our model, we capture the key critical factors that drive the upgrade decisions: increasing functionality requirements due to evolving technology, age-dependent maintenance costs, a predetermined overhaul plan of the asset, and the lifetime of the asset. A system upgrade is less costly if it is executed jointly with an asset overhaul. We first analyze the case with no additional cost of upgrading outside an overhaul. We analytically characterize the structure of the optimal upgrade policy under various realistic assumptions that lead to different types of cost functions. We then use these results as a building block to characterize the optimal policy for a {generalized} cost function. When there is a penalty for upgrading outside an overhaul moment, we propose a dynamic programming approach that efficiently determines the optimal upgrade policy by using our analytical results. We also prove {that as} this penalty increases, the optimal policy can only change to one where the number of upgrades not jointly executed with overhauls is {reduced}. {However, the} optimal number of upgrades is a non-increasing function of this penalty. Also surprisingly, more overhauls may lead to a smaller number of upgrades under the optimal policy. 
\end{abstract}

\section{Introduction}
Maintaining a satisfactory operational level for {physical} assets, i.e.{,} capital goods, during their entire lifetime is a very challenging problem~\citep{Pierskalla1976,Arts2019}. Assets typically have complex designs with many different systems installed. {Each system itself often consists of many components that require replacement in case of failures. The more complex and specialized the systems, the more common it is to see a closer engagement between system suppliers and asset owners, typically in the form of a service agreement that delegates certain roles to the system supplier (e.g., procurement of replacement components, maintenance of the system), referred to as \textit{after-sales services} \citep{li2022after}. Since assets are typically used for a long time, the systems in it can  become obsolete and require upgrades. The decision on when to upgrade a system is a complex one with inputs necessary from both the asset owner and the system supplier. Especially for complex systems customized and serviced by a system supplier in the form of a partnership with an asset owner, there is an interest to make the system upgrade decisions in a collaborative way, and this requires a decision support model to plan the system upgrades as part of after-sales services.} 

{Our work is motivated by a real-life problem as part of a collaboration with an owner of moving assets in a maritime setting (e.g., a frigate or a type of small and fast military ship) and an original equipment manufacturer who designs, builds, and services a complex system in these assets (e.g., a radar system installed on a frigate).} Frigates are typically in use for over 30~years, while the incorporated systems themselves do not necessarily last several decades. {To ensure a frigate works properly,} a schedule is created upfront with designated moments (i.e., overhauls) where the frigate is docked and large-scale maintenance is executed. Docking a ship is costly, and it can therefore be efficient {to perform the upgrade of a system at such an overhaul moment. For instance, by performing an {upgrade of a radar system together} with a frigate overhaul, the ``setup'' required to upgrade the radar system can be saved. On the other hand, the costs for sustaining the current system may increase over time, leading to the question when to perform system upgrades to strike the trade-off between the costs of upgrades and the costs associated with using an existing system in the presence of a predetermined asset overhaul schedule.
}


Based on extensive discussions {with our industry partners}, we came to the {three} dominant factors that drive the costs for sustaining a system: functionality requirements, {the age-dependent behavior of failure frequency and maintenance costs}, and the asset's lifetime. We elaborate on these factors in more detail next. First, advancement of technologies leads to systems that perform much better than the older ones, which may cause higher expectations of a system's functionality. {That is, the asset owner experiences a disutility of not using a system with the most recent technology. Examples include} the desire for a more energy-sustainable system, the wish to rival or even surpass the functionality that competitors have, and more. This aspect is studied in some papers (e.g.{,}~\citealt{Nair1992,Hopp1994}) but is surprisingly often neglected in existing literature on maintenance strategies~\citep{DeJonge2020}.

Secondly, {the components of a system may fail over time, requiring corrective maintenance so that the system can continue its operation as intended.} 
%
%
%
{We note that older systems may require more corrective maintenance, known as age-dependent failure behavior \citep{Wang2002,Zhao2017}. Also,} all maintenance activities come with associated costs, which are stationary in classical variants of age-based maintenance models~\citep{Barlow1960}. Naturally, \textit{age-dependent costs} are often necessary to model reality more closely. Several factors contribute to the need for age-dependent costs, such as salvage value and failure costs. Regarding the salvage value, we observe that the residual value of the system in use typically decreases as newer and improved versions come to market. The failure costs can also be age-dependent, especially when assets have relatively long lifetimes. This is often due to obsolescence issues, i.e.{,} spare components become costlier or even unavailable over time since newer versions are developed and the production of older components or systems is discontinued. Moreover, trained technicians with the knowledge and skills to execute upgrades may no longer be available. A well-chosen strategy can reduce the risk of (unexpected) obsolescence issues and their corresponding costs~\citep{IEC62402}, such as monitoring for obsolescence, last-time-buys~\citep{Behfard2015} or partner agreements, e.g.{,}~incorporation of availability guarantees from suppliers for critical components~\citep{Tomczykowski2003}. Nevertheless, the level of uncertainty and associated costs {make the adoption of these strategies difficult, and can lead to using system upgrades as a means to tackle the increasing  sustainment costs.} 

Finally, the timing of an upgrade may be affected by the remaining lifetime of the asset. Existing literature often focuses on long-run expected average costs for systems with an infinite time horizon. However, an infinite horizon rarely applies to real-life settings: as the end of the asset's lifetime is near, it makes little sense to execute large-scale upgrades or improvements anymore. Moreover, the total costs during the lifetime can be reduced by bringing forward (or postponing) certain maintenance activities.

In this paper, we consider a continuous-time  model that aims to find an optimal upgrade policy for a single system in an asset with a predetermined overhaul planning. 
{The main novelty of our work is to bring all three factors mentioned above in one model in order to address the planning problem of system upgrades as realistic as possible. Also, to the best of our knowledge, our model is the first in the age-based maintenance literature 
	with a schedule of low-cost maintenance moments (i.e., overhaul moments for system upgrades) and age-dependent operating costs (i.e., the age-dependent penalty of not using the system with the most recent technology).} For this model, we analytically characterize the structure of the optimal upgrade policy {by first considering  a special case with no additional penalty of upgrading outside an overhaul moment. We do this first for some special forms of the cost functions, and then, by using these results as a building block, we characterize the structure of the optimal upgrade policy under a {generalized} cost structure.} This allows for a relatively simple characterization of the optimal upgrade policies, and can be used to design an efficient solution approach based on dynamic programming for the case where the penalization {for upgrading outside an overhaul moment} is also included. Finally, we establish sensitivity results with respect to the input parameters. For example, it is intuitively clear that the optimal number of upgrades does not increase as the price for the upgrades increases. However, surprisingly, this is not necessarily true with respect to many other input parameters such as the penalty for upgrading outside overhauls, the number of overhauls, or a more rapid increase of failure costs.

The remainder of this paper is organized as follows. In Section~\ref{sec:Literature}, we review related literature, and in Section~\ref{sec:ModelDescription} we introduce our model formally. In Section~\ref{sec:Analysis}, we derive high-level insights for the structure of the optimal policy if upgrade executions outside overhauls are not penalized. For the general case, we use the results in Section~\ref{sec:Analysis} to introduce a solution approach on how to calculate the optimal policy efficiently in Section~\ref{sec:SolutionApproach}. Finally, we derive several managerial insights in Section~\ref{sec:Sensitivity} by testing the dependency of the optimal policy with respect to the input parameters. We summarize our key findings and insights in Section~\ref{sec:Conclusion}.

\section{Literature review}\label{sec:Literature}
Our work mainly relates to two streams of literature: (i) age-based maintenance models with age-dependent costs, and (ii) upgrade/replacement decision-making by considering the effects of evolving technology. In this section, we review the related literature for both streams, and describe the contributions of our paper.

The first stream concerns age-based maintenance models. In the classical variants of these models, we consider a single system with a stochastic lifetime distribution. The policy maker can choose to execute preventive maintenance against fixed costs that are lower than the costs that are incurred if the maintenance activity would be executed upon failure~\citep{Barlow1960}. However, it is often the case that not all times are convenient to do maintenance on a system. Examples include systems in production lines, the aerospace or maritime industry, or other systems in industries where unplanned downs severely interrupt operations~\citep{Arts2019}. In that case, the policy maker may choose to apply {\it minimal repair} until the planned maintenance moment is reached, meaning that the system is restored to a state right before the failure any time a failure occurs.

There is a vast literature on these types of maintenance models~\citep{Wang2002}. We point out that in these papers executing maintenance is often considered to be a replacement of the system. Initially, stationary costs are considered~\citep{Barlow1960,Barlow1965}, but quickly also age-dependent costs were accounted for in the models{, as shown in Table~\ref{tab:littable}}. In~\cite{Boland1982} and \cite{BolandProschan1982}, a setting is described where the minimal repair costs increase with the number of occurrences of failures for both a finite and an infinite time horizon. A more general minimal repair cost function is considered in~\cite{Tilquin1975}, but replacement costs are assumed to be stationary. Non-decreasing minimal repair and replacement costs are considered in~\cite{Segawa1992} for an infinite time horizon. {Under age-dependent salvage values, \cite{chien2010effect} shows the existence and uniqueness of the optimal age for preventive replacement.} A recent paper~\citep{Sanoubar2020} shows the existence of an optimal replacement policy for an infinite time horizon that allows for a wide class of age-dependent replacement cost functions without minimal repair. {\cite{schouten2022maintenance} extend some classical age-based replacement policies with costs that vary in a cyclical manner.} Literature that considers a finite time horizon is scarcer; \cite{Dagpunar1994}  determine the optimal number of imperfect preventive maintenance actions for a finite horizon given that the minimal repair is made at any failure (with stationary costs).

\begin{table}[h!]
	\centering
	\caption{Literature on age-based maintenance models with age-dependent costs}
	\scalebox{0.95}{
		\begin{tabular}{lccccccccc}
			& \rotatebox{90}{Tilquin and Cleroux (1975)} & \rotatebox{90}{Boland (1982)} & \rotatebox{90}{Boland and Proschan (1982)} & \rotatebox{90}{Segawa et al. (1992)} & \rotatebox{90}{Dagpunar and Jack (1994)} & \rotatebox{90}{Chien (2010)} & \rotatebox{90}{Sonoubar et al. (2021)} & \rotatebox{90}{Schouten et al. (2022)} & \rotatebox{90}{\textbf{Our paper}} \\
			\midrule
			\textbf{Planning horizon } &       &       &       &       &       &       &       &       &  \\
			{\it Finite (deterministic)} &       &       & x     &       &   x    &       &       &       & x \\
			{\it Infinite} & x     & x     & x     & x     &      &   x    & x     & x     &  \\
			\midrule
			\textbf{Time} &       &       &       &       &       &       &       &       &  \\
			{\it Continuous} & x     & x     & x     & x     & x     &  x     & x     &       & x \\
			{\it Discrete} &       &       &       &       &       &       &       & x     &  \\
			\midrule
			\textbf{Preventive action} &       &       &       &       &       &       &       &       &  \\
			{\it Perfect} & x     & x     & x     & x     &       &   x    & x     & x     & x \\
			{\it Imperfect} &       &       &       &       & x     &       &       &       &  \\
			\midrule
			\textbf{Effect of corrective action} &       &       &       &       &       &       &       &       &  \\
			{\it New condition} &       &       &       & x     &       &  x     & x     &       &  \\
			\multicolumn{1}{p{14em}}{{\it Condition just before failure (i.e., minimal repair)}} & x     & x     & x     & x     & x     &       &       &       & x \\
			\midrule
			\textbf{Age-dependent costs} &       &       &       &       &       &       &       &       &  \\
			{\it Minimal repair}  & x     & x     & x     & x     &       &       &       &       & x \\
			{\it Corrective replacement} &       &       &       & x     &       &       & x     & x     &  \\
			{\it Preventive replacement} &       &       &       &       & x     &   x    & x     & x     & x \\
			{\it  Operating cost} &       &       &       &       &       &       &       &       & x \\
			\midrule
			\multicolumn{1}{p{16em}}{\textbf{Schedule of low-cost preventive action moments}} &       &       &       &       &       &       &       &       &  \\
			{\it No}   & x     & x     & x     & x     & x     &   x    & x     & x     &  \\
			{\it Yes}   &       &       &       &       &       &       &       &       & x \\
			\bottomrule
		\end{tabular}%
	}
	\label{tab:littable}%
\end{table}%

This paper considers a continuous-time setting with non-decreasing upgrade costs, failure rate and minimal repair costs, while accounting for the finite lifetime of the asset. {It is important to highlight that the other papers mentioned in Table~\ref{tab:littable} build maintenance optimization models at component level, while we do it at system level (i.e., we take replacement decisions for the system, but at the same time consider a cost function including the maintenance of components within that system)}. A major novel aspect is that not only maintenance on system level is considered, but there already exists an overhaul plan (marked as the schedule of low-cost preventive action moments in Table~\ref{tab:littable}) upfront for the entire asset. This means that the execution of upgrades during the overhauls is typically preferred in order to prevent disruption of asset availability. Upgrades can still be planned at moments when there is no overhaul, but this comes with a  penalty that portrays the severity of the disruption. The effects of such overhaul plans are unexplored in existing literature on age-based maintenance models. {Furthermore, we are the first to consider age-dependent operating costs (to model the disutility of not using the most recent system) within this stream}. 

The second stream of literature relates to the effect of executing upgrades to meet (increasing) functional requirements. This aspect is studied already in a series of papers~\citep{Nair1992,Hopp1994,Nair1995}, where a discrete-time model is proposed in which one needs to decide whether to keep current technology or to move to another version at every time step. Through a Markov Decision Problem formulation, an optimal solution can be determined through dynamic programming, and also the infinite horizon case can be solved through an efficient algorithm. Technology advancement is also considered in~\cite{Rajagopalan1998b}, but for an entirely different environment where technologies provide a certain capacity that is needed for operations. {In this stream, it is common to assume an underlying stochastic model to represent the evolution of technology or the prediction of technological advancements. This is not necessary in our work, where we directly model the disutility of not using the system with the most recent technology.} A continuous-time multi-component maintenance model is considered in~\cite{Mercier2008}, where the goal is to identify when to upgrade all components after a single technology improvement among the class of reactive upgrade policies. Analytical results are derived for a similar problem in~\cite{Oner2015}, but then from the perspective of a maintenance service provider who has to choose between preventive and corrective upgrading for a fleet of systems/assets after a redesign of a component. \cite{Nguyen2013} bridge technology change with spare part inventory management in a discrete-time model, and use dynamic programming to determine the optimal maintenance policy for a fixed horizon. Although this paper is rich in the different aspects that come with upgrades (i.e.{,} potential loss of spare parts, different maintenance options, and so forth), no structural insights on the optimal policy are provided. 

In this paper, we consider a continuous-time  model similar as in~\cite{Mercier2008}. We restrict ourselves to an upgrade policy for a single system, but we allow for continuously evolving technology and preventive upgrade policies (with a minimal repair strategy to deal with failures). Our underlying costs and penalty functions are age-based, and we derive structural insights for the optimal upgrade policy. In addition, we incorporate overhaul moments where large-scale maintenance is executed, and hence it can be cost-effective to execute a system upgrade at such a moment.

\section{Model description}\label{sec:ModelDescription}
We consider an asset that is in operation for a finite (known) lifetime~$H$. Typically, an asset has many different systems, subsystems and components in place on different indenture levels. We take a simplified view as illustrated in Figure~\ref{fig:AssetDecomposition}. That is, an asset is composed of different systems, and in turn, every system has many different (critical) components that are needed for an operational system. {In practice, the upgrade decision of each system can be decoupled from each other. For example, a frigate as an asset has various systems (e.g., propulsion, navigation, radar, weapon systems) that work independently. That is, the failures of a navigation system do not depend on the failures of a radar system. Also, each system is designed and serviced by a different supplier. Thus, the penalty of not using the most recent version of these systems can be modeled separately in collaboration with different suppliers. Therefore, we aim} to determine a cost-optimal upgrade planning for a single system. We assume that all costs are expressed in monetary units at the start of the planning horizon (so we do not compute net present values). We provide a formal description of our model and notation next. We summarize our notation in Appendix~\ref{app:Notation} as well.

\begin{figure}[htb!]
	\centering
	\caption{Asset decomposition.}
	\label{fig:AssetDecomposition}
	\begin{tikzpicture}[scale=0.8]
		
		\draw (-1,10) rectangle (1,9);
		\draw (0,9.5) node {asset};
		
		\filldraw[fill=black!10!white, draw=black] (-2.5, 8) rectangle (-0.5, 7);
		\draw (-1.5,7.5) node {\small system};
		\draw (-5.5, 8) rectangle (-3.5, 7);
		\draw (-4.5,7.5) node {\small system};
		\draw (2.5, 8) rectangle (4.5, 7);
		\draw (3.5,7.5) node {\small system};
		\draw[thick,dotted]  (0,7.5) -- (2,7.5);
		
		\draw (0,9) -- (-1.5,8);
		\draw (0,9) -- (-4.5,8);
		\draw (0,9) -- (3.5,8);

		\draw (-2.5, 6) rectangle (-0.5,5);
		\draw (-1.5,5.5) node {\scriptsize component};
		\draw (-5.5, 6) rectangle (-3.5, 5);
		\draw (-4.5,5.5) node {\scriptsize component};
		\draw (2.5, 6) rectangle (4.5,5);
		\draw (3.5,5.5) node {\scriptsize component};
		\draw[thick,dotted]  (0,5.5) -- (2,5.5);
		
		\draw (-1.5,7) -- (-4.5,6);
		\draw (-1.5,7) -- (-1.5,6);
		\draw (-1.5,7) -- (3.5,6);
		
	\end{tikzpicture}
\end{figure}
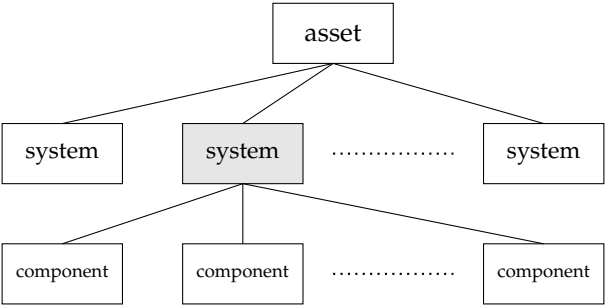

\newpage
The upgrade of a system may interfere with the operations of the entire asset. {It is therefore efficient to perform the upgrade of a system together with other asset-level large-scale maintenance activities, referred to as overhauls. For example, when a frigate is already out-of-service for overhaul, any intervention to the frigate can be performed at no additional cost of taking the frigate out of service (which would be necessary during an upgrade).}  We assume that there exists an overhaul plan with planned moments for the entire asset. 
Overhaul periods are relatively short with respect to the total lifetime of the asset. Therefore, we allow for a simplification of the model by assuming that the overhaul takes no time in the planning, and hence the state of the system also has not changed after an overhaul unless an upgrade had been planned. This yields a continuous-time model with given overhaul moments. 

We denote the time between consecutive overhauls by $M_1,\ldots,M_m, M_{m+1}$, where $m$ is the number of overhauls during the asset's lifetime {$H$}. In particular, $M_1$ denotes the time of the first overhaul, and $M_{m+1}$ denotes the duration between the last overhaul and the end of the asset's lifetime. We allow for system upgrades to be executed at times when there is no overhaul, but we penalize these events with fixed costs~$c_d$. Note that if a system upgrade can be executed easily and has no asset-level effect, the setting simplifies to $c_d=0$. We refer to this special setting as the {\it base case}.

As mentioned in the introduction, we model several factors that drive an asset owner to upgrade a system through generalized costs and penalty functions. We explain this in more detail next. First, we say that a system upgrade comes with certain upgrade costs, for which we assume the following.

\begin{assumption}\label{ass:UpgradeCosts}
	The upgrade costs can be written as $c_0-v(t)$, where $c_0 > 0$ can be seen as the price for purchasing and installing the upgraded version of the system, $v(t)$ corresponds to the salvage value of the current system, and $t \geq 0$ is the time since the current system is in use. Moreover, we assume that $c_0 > v(0)$, and that the salvage value $v(t)$ is non-increasing and differentiable in~$t$. 
\end{assumption}

We point out that the price for a system upgrade $c_0$ is constant {over time}, i.e.{,} has no dependency on time. This independence reflects that asset owners are willing to pay a certain fixed price for the newest generation of a system, which is common for capital goods \citep{chien2010effect}. The asset owner may not be willing to pay a higher price than what was paid for an older system with the same asset functionality as when it was first used. From the seller's (e.g., the system supplier who designs, produces and services the system) point of view, increasing the price of a system along the way may lead to losing the account with the customer, or be infeasible due to contractual agreements. The salvage value $v(\cdot)$ corresponds to the possible costs or returns that come with the disposing of the current system in use, and depends only on the time it has been in use. In particular, we make no assumptions on whether the salvage value is positive or negative. We exclude the unrealistic scenarios where $c_0 \leq v(0)$, otherwise it would be profitable to continuously upgrade. The property that the salvage value is non-increasing describes the natural phenomenon that the disposal of an older system becomes less profitable (or costlier) over time. Also in cases where equipment is refurbished or reused, the salvage value can decrease in age since the refurbishment process may involve adjustments and additions of extra components~\citep{Sanoubar2020}. 

As soon as a certain system is installed in the asset, improved versions and upgrades may come to the market. Increasing functional requirements of the current system may lead to the necessity for upgrading to newer versions to meet the desired requirements. We model this effect through a penalty function $c_f(t)$, where $t$ denotes the time that the system is in use since its last upgrade. It represents the gap in functionality between the system in use and the upgraded version, see e.g.{,}~\cite{Sols2012} for an example of such a gap function. Note that by definition there is no gap when the system in use is brand-new, i.e.{,}~$c_f(0)=0$. 

\begin{assumption} \label{ass:penalty}
	The functionality gap function $c_f(t)$ is non-decreasing in $t$ with $c_f(0)=0$, where $t$ denotes the time since the current system is in use, i.e., the last upgrade.
\end{assumption}

{In practice, the penalty function $c_f(\cdot)$ can be directly elicited from the system supplier (who may assess the condition of the latest technology and current capabilities) and the asset owner  (who may assess how critical it is for the entire asset to exploit the benefits of the latest technology in a particular system). For example, the disutility of not using the system with the most recent technology can be a steeply increasing function of time if the system is safety critical (e.g., a software module protecting the asset from failures but becoming prone to cyberattacks over time), or it can be a constant or slowly increasing function if it performs the defined function adequately regardless of its age (e.g., a furnace used for the heat-treatment of steel products). In some cases, historical data can be used to derive $c_f(\cdot)$. For instance, consider a piece of machining equipment, which requires a certain amount of power to perform its task when it is new. The wear caused by its usage may lead to an increase in the cutting force required in the process, and hence, the required cutting power. So, there is an influence of how long the machine has been used on energy consumption. Historical data can be used to identify a function that represents how the power required by the process changes with respect to the usage hours. Such a function can be translated into a penalty function $c_f(\cdot)$ that quantifies the additional energy cost as a function of time. 
}

The system itself consists of many different (critical) components that are subject to failures. We aggregate the failures to a single failure rate for the system, where every failure comes with a certain expected cost. We consider the setting where failures do not occur less likely as time progresses, as well as that the associated costs do not lessen over time. Moreover, we assume that the repair or replacement of a component restores the total failure rate back to the value right before the failure event, also known as a minimal repair strategy~\citep{Barlow1965}.

\begin{assumption}\label{ass:FailureCosts}
	Let $h(t)$ denote the failure rate of the system, and $k(t)$ the expected repair costs to bring the system back to the state just prior to the failure, where $t \geq 0$ denotes the time since {the current system is in use}. The functions $h(t)$ and $k(t)$ are both non-decreasing in $t$. Repairs are immediate.
\end{assumption}

We point out that minimal repair is assumed due to the notion that the system is composed of many components and the failure of one critical component will not (likely) affect the total failure rate of the system{, as also argued in \cite{asadi2022overview}}. {For a particular system that is always sold at the same price, it can be argued that the quality level (i.e., reliability) of the systems will be similar across generations, hence, we assume the failure rate function $h(\cdot)$ remains the same for each system}. {There is an extensive literature in reliability theory to estimate failure rate functions by using historical data (see, e.g., Ch. 3 of  \citealt{gertsbakh2000reliability}). We note that the failure rate function is best determined with the collaboration of the asset owner (who knows how intensively the system is used) and the system supplier (who has the expertise of system designers and/or can pool data from many different users of their systems).} The associated expected costs that come with a failure have a general form, as long as it is non-decreasing in time. These costs are an aggregation of the actual repair and replacement costs and downtime costs, and particularly obsolescence issues for critical components. That is, the non-decreasing property of $k(t)$ mostly stems from obsolescence issues that can arise when critical components are no longer on stock nor produced, and the repair becomes costlier as a more expensive fix needs to be applied.

\begin{remark}
	For a system with many critical different components, it may hold that the failure of the system relates to the failure of one particular component by a certain probability. In this example, the expected repair costs $k(t)$ may be written as
	\begin{align}
		k(t) = \sum_{i=1}^I k_i(t) p_i,
	\end{align}
	where $I$ is the number of (critical) components in the system, $p_i$ is the probability that the failure is due to the breakdown of component~$i$, and $k_i(t)$ is the expected repair cost of a failed component~$i$. {Note that the values of $p_i$ can be determined by using historical data of component failures, and $k_i(t)$ by using the costs associated with component procurement and repairs at different ages of the system, which can be obtained by, e.g., monitoring the spare part prices on the market.} 
\end{remark}

The goal is to find the policy that minimizes the total costs, which is the aggregation of all costs and penalties over the entire asset's lifetime. To formalize this, we introduce the following notation. Under a given policy $\Pi$, let $N$ denote the number of upgrades, $T_1$ the first upgrade moment (if $N>0$), and $T_2,\ldots,T_N$ the time between upgrade moments (if $N\geq 2$). Write $T_{N+1}=H-\sum_{i=1}^N T_i$ as the time between the last system upgrade (if any) and the end of the asset's lifetime, which we also refer to as the remaining lifetime. Let a \textit{cycle} be the time a single system version is in use, i.e.{,} the time between the moment the system is installed till it is disposed off, which is either at the next upgrade moment or at the end of the asset's lifetime.  We define the \textit{cycle costs} $C(T)$ as the aggregation of the salvage value, missing functionality penalty, and the expected costs that come with system failures in a cycle of length $T$, i.e.{,}
\begin{align}\label{eq:CycleCosts}
	C(T) = -v(T) + \int_0^{T} c_f(t) \, dt + \int_0^{T} k(t) h(t) \, dt.
\end{align}
We point out that the price $c_0$ is not included in the cycle costs as the last cycle covers a disposal but not an upgrade. Under policy~$\Pi$, let $S^\Pi$ denote the number of times a system upgrade is executed but not jointly with an overhaul, i.e.{,}
\begin{align*}
	S^\Pi = \left\vert \left\{ n \leq N : \sum_{i=1}^n T_i \neq \sum_{i=1}^j M_i \text{ for all } 1 \leq j \leq m \right\} \right\vert.
\end{align*}
Under a given policy $\Pi$, the total costs $\mathcal{K}^\Pi(H)$ over a lifetime $H$ are given by
\begin{align}\label{eq:TotalCostsGeneral}
	\mathcal{K}^\Pi(H) = S^\Pi c_d + N c_0 + \sum_{i=1}^{N+1} C(T_i) .
\end{align}
{Note that $N,T_1,\ldots,T_N$ are the decision variables in our model.} The objective is to minimize the total costs, which we denote as
\begin{align}
	\mathcal{K}^*(H) = \min_{\Pi} \mathcal{K}^\Pi(H)
\end{align}
with corresponding optimal policy $\Pi^*$ specified by the number of upgrades $N^*$ and the inter-upgrade times $T_i^*, i=1,\ldots,N^*$. 

If $c_d=0$, we change the total costs variable from $\mathcal{K}$ to $\mathcal{C}$ to stress the fact we consider the base case. Under a given policy~$\Pi$, the total costs reduce to
\begin{align}\label{eq:TotalCostsBaseCase}
	\mathcal{C}^\Pi(H) = N c_0 + \sum_{i=1}^{N+1} C(T_i).
\end{align}
We denote $\mathcal{C}^*(H)$ as the minimal total costs in the base case. 

A central question is whether we can derive results on what the optimal policy looks like structurally. A straightforward observation is that it can never be optimal to upgrade more than a finite number of times.

\begin{lemma}
	Let $\bar{N} =(C(H)+v(0))/(c_0-v(0)) \in [0,\infty)$. Any policy $\Pi$ with $N > \bar{N}$ upgrades satisfies $\mathcal{K}^\Pi(H) > \mathcal{K}^*(H)$.
	\label{lem:FiniteNumberOfReplacements}
\end{lemma}

A short proof of Lemma~\ref{lem:FiniteNumberOfReplacements} can be found in Appendix~\ref{app:proofs}. As a consequence, Lemma~\ref{lem:FiniteNumberOfReplacements} indicates certain scenarios where it is optimal to never upgrade a system during the asset's lifetime. This typically applies to expensive systems with little technological advancements and which are robust to failures.

\begin{corollary}
	If the upgrade costs are very high, it is never optimal {to upgrade. Specifically}, this is the case if $c_0 -2v(0) \geq C(H)$.
\end{corollary}

Next to the upgrade costs, we can identify two extreme cases with respect to the penalty~$c_d$. The first extreme case is when an upgrade disrupts the asset's operation completely, and an extremely high penalty $c_d$ is incurred. In that case, it would only be rewarding to upgrade during overhauls, which reduces the number of possible upgrade plans significantly. One needs to only decide at the $m$ overhaul moments whether to upgrade or not, where $m$ is finite and relatively small. This can be solved through dynamic programming, which we discuss in more detail in Section~\ref{sec:SolutionAppOnlyOverhauls}.

The other extreme case concerns the base case where $c_d=0$, i.e.{,} no penalty is given to doing upgrades if it does not coincide with an overhaul. This corresponds to systems that can be easily maintained and upgraded without affecting the total operation mode of the asset. It turns out that the optimal policy has a simple structure for a large class of cycle cost functions. We consider this in more detail next.

\section{Analysis of the base case}\label{sec:Analysis}
In this section, we assume that $c_d=0$, i.e., that the presence of overhauls for the asset is not relevant, and determine the optimal upgrade policy for different shapes of the cycle cost function. We use these results as building blocks to design an efficient algorithm to find the optimal policy for the general case where also positive values of the penalty $c_d\geq0$ are allowed. Therefore, we spend significant effort to determine the structure of the optimal upgrade policy in the base case{.}
{The proofs of the results in this section can be found in Appendix~\ref{app:proofsBaseCase}.}

Recall that the total cost expression simplifies to~\eqref{eq:TotalCostsBaseCase} in the base case for a given policy~$\Pi$. In view of Lemma~\ref{lem:FiniteNumberOfReplacements}, we can try to derive for every $N \in \mathbb{N}_{\geq 0}$ with $N< \bar{N}$ what the corresponding optimal inter-upgrade times are. If we can answer this question, we can simply compare the total costs for a finite number of upgrades to derive the optimal upgrade policy.

Mathematically, this means that our problem can be rewritten as 
\begin{align}
	\mathcal{C}^*(H) = \min_{0 \leq N \leq \lfloor \bar{N} \rfloor}\left\{ N c_0 +  \min_{T_1+\ldots+T_{N+1}=H} \left\{ \sum_{i=1}^{N+1} C(T_i) \right\} \right\}.
	\label{eq:TotalCostsBaseConvexInitialProblem}
\end{align}

\noindent 
For a fixed $N$, we can determine the optimal upgrade policy by deriving the inter-upgrade times that minimizes 
\begin{equation}
	\sum_{i=1}^{N} C(T_i)+C(H-\sum_{i=1}^{N} T_i),
	\label{eq:obj}
\end{equation}
where $T_i \in [0,H]$ for all $i=1,\ldots,N$. Over the finite interval {$[0,H]$ and for a fixed $N$}, a minimum of \eqref{eq:obj} can only be attained at points where each partial derivative (in $T_i$) is either zero or at the boundaries. 
{However, we know the latter is not possible as the system upgrade costs more than salvaging the system immediately (i.e., for the optimal upgrade policy, it cannot hold that $T_i^*=0$ for some $i=1,\ldots,N$ or $\sum_{i=1}^{N^*} T_i^*=H$ since $c_0>v(0)$ by Assumption~\ref{ass:UpgradeCosts}).} 
{Therefore, an optimal upgrade policy} must satisfy the property that every partial derivative (in $T_i$) is zero. That is, a necessary condition for optimality is given by
\begin{align}
	C'(T_i^*)= C'\left(H-\sum_{j=1}^N T_j^* \right), \hspace{1cm} \forall i=1,\ldots,N^*,
	\label{eq:NecessaryConditionC}
\end{align}
where $C'(t)$ denotes the first derivative of the cycle cost function $C(t), t\geq 0$. To find the optimal upgrade policy, we are therefore interested in the sets
\begin{align}\label{eq:CandidateSolutionsNUpgrades}
	\mathcal{T}^N = \left\{ (t_1,t_2,\ldots,t_N) \in \mathbb{R}_{>0}^N : \sum_{j=1}^N t_j < H , C'(t_i)= C'\left(H-\sum_{j=1}^N t_j\right), \;\;\; i=1,\ldots,N \right\}, \;\;\; N\geq 1.
\end{align}
These sets represent the candidates for the optimal inter-upgrade times if $N \geq 1$ upgrades would be executed during the lifetime. The optimal policy is therefore contained in the set
\begin{align}
	\mathcal{T} = \bigcup_{N=0}^{\lfloor\bar{N} \rfloor} \mathcal{T}^N,
\end{align}
where $\mathcal{T}^0 = \emptyset$ represents the case where no upgrades are executed.

The size of the sets~$\mathcal{T}^N$ depends on the behavior of the cycle cost function $C(\cdot)$. When we take a closer look {at} our assumptions, we observe that more information can be derived on the behavior of the cycle costs. Recall that the cycle costs are given by~\eqref{eq:CycleCosts}, and the first derivative is thus given by 
\begin{align*}
	C'(T) = -v'(T) + c_f(T) + k(T) h(T) \geq 0.
\end{align*}
{Assumptions~\ref{ass:penalty} and \ref{ass:FailureCosts}} imply that $c_f(T) + k(T) h(T)$ is non-negative and non-decreasing in $T$, and hence this gives rise to a convex term in the cycle cost function. {This means that the cycle cost function $C(\cdot)$} is an aggregation of a convex non-decreasing function, and a general non-decreasing function $-v(T), T\geq0$. Naturally, the salvage value function $v(\cdot)$ can take different forms, leading to a certain behavior for the cycle cost function $C(\cdot)$. In our analysis, we first focus on the special cases where the cycle costs are convex (Section~\ref{sec:ConvexCycleCostsBase}), concave (Section~\ref{sec:ConcaveCycleCostsBase}) or S-shaped (Section~\ref{sec:SShapeCycleCostsBase}). By building on these results, we then consider a generalized structure for the cycle cost function and establish the optimal upgrade policy structure in Section~\ref{sec:GeneralCycleCosts}.

\subsection{Convex cycle costs}\label{sec:ConvexCycleCostsBase}
In this section, we consider the case where the salvage values are such that the cycle costs are convex and non-decreasing. {In practice, there are many real-life scenarios that fall under this case. An example is the case with a constant salvage value, which occurs when a system cannot be resold after use or when the disposal of an older version comes with some cost that is not affected by time.} This setting leads to convex cycle costs. {In some other settings, where reselling is an option, initially the resell value can be  close to the initial price, but as newer technologies come to the market, the resell value decreases rapidly at a non-decreasing speed (i.e., the speed of reduction in the resell value can be constant, leading to a linearly decreasing salvage value; or it can be be increasing, leading to a concave salvage value). In general, we can say a non-increasing concave salvage value function $v(t)$ leads to convex non-decreasing cycle costs. Even if the decrease in the salvage value slows down over time, as long as the moment at which this slowing down occurs before the end of the lifetime of the asset, the cycle cost function continues to be a convex and non-decreasing function.} 




If the cycle costs are convex and non-decreasing, it turns out that an optimal upgrade policy for the base model has a  simple structure. Specifically, an optimal strategy to follow is to either never upgrade, or to upgrade a finite number of times $N^*$ with equidistant inter-upgrade times $T_i^* = H/(N^*+1)$, $i=1,\ldots,N^*$. Intuitively, it is clear that it is never optimal to upgrade more than a finite number of times since every upgrade always comes with a certain strictly positive cost. Equidistant inter-upgrade times is a property that comes from the convexity of the cycle costs. {This can be illustrated with a simple example as follows. Suppose that there will only be one upgrade during the asset's lifetime, and this upgrade takes place at time $T$}. Then the total cost is given by $c_0+C(T)+C(H-T)$. The cycle costs are convex, and by definition, this implies that for every $\Delta \in [0,H/2]$,
\begin{align*}
	c_0 + C\left(\frac{H/2-\Delta}{2} \right) + C\left(\frac{H/2+\Delta}{2} \right) \leq  c_0+ 2  C\left(\frac{H/2-\Delta}{2} + \frac{H/2+\Delta}{2}\right) = c_0 + 2 C\left(\frac{H}{2}\right),
\end{align*}
where an equality holds if $\Delta=0$. Therefore, the {total cost is} minimized if $T=H/2$. Similarly, this also holds in general when executing multiple upgrades, {as formalized in Proposition~\ref{prop:EquidistantResult}}.  

\begin{proposition}
	Suppose that the cycle costs are convex. Then, it is either optimal to never upgrade, or to do so for $N^* \in \mathbb{N}$ times with $N^* \leq \bar{N}$. In the latter case, $T_i^*=H/(N^*+1)$, $i=1,\ldots,N^*$ and the minimal total costs are given by
	\begin{align}
		\mathcal{C}^*(H) = \min_{N \in \{0,1,\ldots,\lfloor\bar{N}\rfloor\}} \left\{ Nc_0+(N+1) C\left(\frac{H}{N+1}\right) \right\}.
		\label{eq:TotalCostsForConvexBase}
	\end{align}
	\label{prop:EquidistantResult}
\end{proposition}
{Note that $\bar{N}$ is the upper bound on the optimal number of upgrades as characterized in Lemma~\ref{lem:FiniteNumberOfReplacements}. Proposition~\ref{prop:EquidistantResult} shows that the problem of finding the optimal upgrade times is equivalent to the problem of finding the optimal number of (equidistant) upgrades. That is, the optimal upgrade policy can be found in practice by identifying the value of $N \in \{0,1,\ldots,\lfloor \bar{N}\rfloor \}$ that minimizes $\mathcal{C}^N(H)$ for a given asset lifetime $H$, where}
\begin{align}\label{eq:TotalCostsBaseNUpgrades}
	\mathcal{C}^N(H) = Nc_0+(N+1) C\left(\frac{H}{N+1}\right). 
\end{align}
\noindent



In case that {the salvage value of a new system is close to the upgrade price (i.e.,} $c_0-v(0)$ is small{), the value of $\bar{N}$ can be high. Thus, it may be necessary to check a large number of possible $N$ values to be able to confirm the optimality}. 
{In Lemma~\ref{lem:ConvexityOfCN}, we show that the function $C^N(H)$ has a discrete convex structure in $N$ for a given asset lifetime $H$.} 

\begin{lemma}
	If the cycle cost function is convex, then the function $\mathcal{C}^N(H)$ is {discrete} convex in $N$ {for a given $H$}. 
	\label{lem:ConvexityOfCN}
\end{lemma}

{For an asset with lifetime $H$, the discrete convex structure of $C^N(H)$ in $N$ can be exploited to speed up the process of searching for the optimal number of upgrades~$N^*$ by providing an optimality guarantee. To be specific,} Proposition~\ref{prop:EquidistantResult} and Lemma~\ref{lem:ConvexityOfCN} imply that in order to find an optimal number of upgrades~$N^*$, we need to determine the final integer $N$ before $\mathcal{C}^N(H)$ increases. 

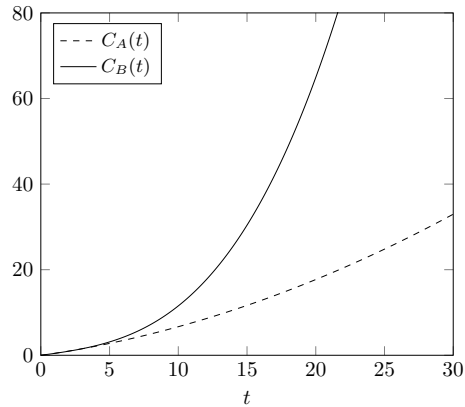
\begin{figure}[htb!]
	\centering
	\caption{Cycle costs for $t\in[0,H]$.}
	\label{fig:Example1Functions}
	\begin{tikzpicture}[scale=0.8]
		\begin{axis}[domain=0:30, xmin=0, xmax=30,ymin=0,ymax=80,samples=300,legend style={at={(1.1,1)}, anchor=north west},legend pos=north west,no
			marks,legend entries={\small $C_A(t)$,\small $C_B(t)$},xlabel=$t$]
			
			\addplot[dashed] {\x/3+3/16*(\x/3)^2 + 1/10*\x^(1.1)};
			\addplot[] {\x/3+3/16*(\x/3)^3 + 1/10*\x^(1.1)};
			
		\end{axis}
		
	\end{tikzpicture}
\end{figure}

\begin{example}\label{ex:MainExample}
	Consider two settings, where in both settings $H=30$ years and $c_0=4$. In setting~A, let the cycle costs be given by
	\begin{align}
		C_A(t) = \frac{t}{3} + \frac{3}{16} \left(\frac{t}{3}\right)^2 + \frac{1}{10}
		t^{1.1},
	\end{align}
	and in setting~B, let
	\begin{align}
		C_B(t) = \frac{t}{3} + \frac{3}{16} \left(\frac{t}{3}\right)^3 + \frac{1}{10}
		t^{1.1}.
	\end{align}
	In Figure~\ref{fig:Example1Functions}, we plot the cycle costs as a function of time.
	
	Note that the cycle costs are strictly convex on $[0,H]$ for both settings. In view of Proposition~\ref{prop:EquidistantResult} and Lemma~\ref{lem:ConvexityOfCN}, we can conclude that the (unique) optimal policy is to upgrade $N^*$ times after every $H/(N^*+1)$ years, where $N^*$ is the final integer before $C^N(H)$ increases. Using Table~\ref{tab:FirstIllustrativeExample}, we observe that it is optimal to execute an upgrade once ($N_A^*=1$) after 15 years in setting~A. In setting~B, it is optimal to upgrade $N_B^*=4$ times after every 6 years.
	
	\begin{table}[htb!]
		\centering
		\caption{Total costs values for setting A and B.} 
		\label{tab:FirstIllustrativeExample}
		\begin{tabular}{ | l | l | l |}
			\hline
			$N$ & $\mathcal{C}^N_A(H)$ & $\mathcal{C}^N_B(H)$\\ 
			\hline
			0 & 32.9653 & 201.7153 \\
			1 & 27.3081 & 64.8081 \\
			2 & 28.0268 & 42.6101 \\
			3 & 30.3572 & 37.3884 \\
			4 & 33.3387 & 37.0887 \\
			5 & 36.6489 & 38.7322 \\
			\hline
		\end{tabular}
	\end{table}

\end{example}

\subsection{Concave cycle costs}\label{sec:ConcaveCycleCostsBase}
In this section, we consider the case {where the salvage values are such that the cycle costs are concave and non-decreasing.  In practice, there can be settings where the salvage value of a new system rapidly decreases as soon it is installed to the asset, and the speed of decrease in the salvage value can slow down over time (this is the direct opposite of the case in Section~\ref{sec:ConvexCycleCostsBase}), leading to a convex non-increasing salvage value function $v(t)$. Recall that the first derivative of the cycle cost function $C(t)$ is given by $C'(t)=-v'(t) + c_f(t) + k(t) h(t)$. That is, when  the rate of increase in the costs associated with the functionality gap and the component failures over time (captured by the term $c_f(t) + k(t) h(t)$) is small relative to the rate of loss in the salvage value, we observe  concave cycle cost functions. For example, consider a system that has a convex non-increasing salvage value with no functionality gap penalty (i.e., $c_f(t)=0$). This may be the case when the operator of the asset only requires an operating system but no more than that, even if the technology increases and there are more efficient systems in the market. If such a system has components that fail at a constant rate (e.g., electronic components) and can be repaired at similar costs regardless of the age of the system, we end up with concave cycle costs.}


For concave cycle cost functions, it is optimal to never upgrade. Intuitively, this statement can be explained as follows. If it would be optimal to execute at least one upgrade during the lifetime, then due to the concavity of the cycle cost function, an optimal strategy would be to immediately execute all upgrades at the start of the asset's lifetime. But since an upgrade always comes with a positive cost, this implies it is best to never upgrade at all. {We formalize this result in Lemma~\ref{lem:ConcaveNeverUpgrade}.}

\begin{lemma}\label{lem:ConcaveNeverUpgrade}
	Suppose that $C(t)$ is concave for all $t \in [0,H]$. Then, it is optimal {to never upgrade}.
\end{lemma}

\subsection{S-shaped cycle costs}\label{sec:SShapeCycleCostsBase}
{Section~\ref{sec:ConvexCycleCostsBase} considers convex and non-decreasing cycle cost functions, where the increase in the costs grows over time. Afterward, Section~\ref{sec:ConcaveCycleCostsBase} considers concave and non-decreasing cycle cost functions, where the increase in the costs gets slower over time. While there are realistic settings where each behavior is plausible, there are also cases where the cycle cost function shows a combination of these behaviors.}
Suppose that the increase in cycle cost grows fast at the early phases of the life of the system, then the growth stabilizes, and finally, the growth becomes slower. This behavior for the cycle cost is referred to as an S-shaped function.
In this section, we consider the case that the cycle cost function is S-shaped, i.e.{,} there exists a point of inflection $x \in [0,H]$ such that the cycle costs are convex on the interval $[0,x]$ and concave on $[x,H]$. 


For a typical scenario where S-shaped cycle costs may emerge, one can think of an electronic device for which upgraded versions with technological improvements appear on the market over time, making the older version less valuable. The salvage value initially decreases slowly, but the drop increases as newer versions appear. This effect is followed by a period where the drop in the salvage value gets smaller again. In other words, the salvage value function of the electronic device has a reversed S-shape. Depending on the behaviors of the missing functionality penalty and expected failure costs, this in turn may result in S-shaped cycle costs.

\begin{remark}
	The point of inflection $x \in [0,H]$ does not necessarily have to be uniquely defined. For example, if the cycle costs are linearly increasing on $[a,b] \subset [0,H]$ with $a<b$, strictly convex on $[0,a]$ and strictly concave on $[b,H]$, then any $x \in [a,b]$ can be chosen as the point of inflection. Yet, in the results that are provided in this section, we require the point of inflection $x$ to be chosen in a (unique) specific manner. That is, we say that $x$ is chosen such that 
	\begin{align}\label{eq:TecnicalRequirement}
		C(x+\Delta)-C(x) < C'(x) \Delta, \hspace{2cm} \forall \Delta>0.
	\end{align}
	For the mentioned example, this implies that the point of inflection is chosen to be equal to~$b$. In the remainder of this paper, points of inflection that satisfy this property~\eqref{eq:TecnicalRequirement} will be referred to as satisfying the \textit{technical requirement}.  
\end{remark}

It turns out that in the case of S-shaped cycle costs, there is an optimal policy that has a relatively simple structure. It has one of the following structures:
\begin{enumerate}
	\item It is optimal to never upgrade;
	\item It is optimal to upgrade $N^* \geq 1$ times after every $H/(N^*+1)$ time units;
	\item It is optimal to upgrade $ N^* \geq 1$ times after every $T  < H/(N^*+1)$ time units, where $T$ differs from the time between the final upgrade and the end of the lifetime.
\end{enumerate}

The reason that this holds already follows implicitly from the results in the two previous sections. Suppose there is an optimal policy where we upgrade $N^*=N_1+N_2-1$ times, where $N_1$ and $N_2$ are the number of times we have inter-upgrade times in the interval $[0,x]$ and $(x,H]$, respectively (where we slightly abuse definitions by including the time between the last upgrade and the end of the lifetime as an inter-upgrade time). Since the cycle costs are concave on $(x,H]$, {it follows from our analysis in Section~\ref{sec:ConcaveCycleCostsBase} that} whenever $N_2 \geq 2$, there exists an optimal policy where at least one of the $N_2$ inter-upgrade times lies at the boundary of the interval $(x,H]$. This is not possible, and therefore it must hold that $N_2 \leq 1$. In particular, if $N_2=1$, without loss of generality, we can set the corresponding inter-upgrade time as the time between the final upgrade and the end of the lifetime, i.e.{,}~as $T_{N^*+1}$. Moreover, {it follows from our analysis in Section~\ref{sec:ConvexCycleCostsBase} that} there exists an optimal upgrade policy where the $N_1$ inter-upgrade times in interval $[0,x]$ equal one another (if any) since the cycle costs are convex on this interval. 

In order to prove this formally, we require some additional results on the optimal policy structure, {see Appendix~C for details.} {To be specific, we first show in Lemma~\ref{lem:SShapedPropertyUpgradeTime} that, if $N^*\geq 1$,} there exists an optimal policy where no inter-upgrade time lies strictly between $\min\{x,H/(N^*+1)\}$ and $\max\{x,H/(N^*+1)\}$. {In Lemma~\ref{lem:NotOnlyUpgradeAfterTurningPoint}, we then show} there exists an optimal policy that always contains at least one inter-upgrade time that is at most equal to the point of inflection $x$ of the cycle cost function{, i.e., $T_i^*\leq x$ for some $i=1,\ldots,N^*+1$.}
%
%
%
{By using these results, Proposition~\ref{prop:OptimalUpgradePolicyBase} shows the main result of this section, which gives insight {into} the structure of the optimal upgrade policy.}

\begin{proposition}
	Let the cycle costs have an S-shape with point of inflection $x \in (0,H)$ satisfying the technical requirement. It is either optimal to execute no upgrades, or there exists an optimal policy with $N^*\geq 1$ upgrades that has one of two structures: either $T_1^*=\ldots=T_N^* = H/(N{^*}+1) = T_{N{^*}+1}^*$, or $T_1^*=\ldots=T_N^* \leq x <T_{N{^*}+1}^*$. 
	\label{prop:OptimalUpgradePolicyBase}
\end{proposition}

{Besides providing an analytical insight on the structure of the optimal upgrade policy, Proposition~\ref{prop:OptimalUpgradePolicyBase} is useful as it establishes that the set of optimal upgrade policies can be} simplified to the union of the set $\mathcal{T}^0$ (no upgrading) and the sets
\begin{align}
	\mathcal{T}^{N_1,N_2} = \left\{ (t_1,t_2) \in \mathbb{R}^2 : 0<t_1 \leq x < t_2 <H, N_1 \in \mathbb{N}, N_2 \in \{0,1\}, N_1 t_1 +N_2 t_2 = H , C'(t_1)= C'(t_2) \right\},
	\label{eq:OptimalCandidatesSShape}
\end{align}
where $N_1$ is bounded due to Lemma~\ref{lem:FiniteNumberOfReplacements}. To illustrate how this result can be used to determine an optimal upgrade policy, we consider the following example.

\begin{example}
	\label{exmS}
	Suppose that $H=30$, $c_0\geq 1$, $c_f(t)+k(t)h(t)= 0$ and
	\begin{align*}
		C(t) = -v(t) = -1+\left( 1+\exp^{-({t-10})} \right)^{-1}. 
	\end{align*}
	This is an S-shaped function with $C(t)=-v(t) \in (-1,0)$ for all $t \in [0,H]$ and (unique) point of inflection $x=10$. Note that
	\begin{align*}
		C'(t)=-v'(t) =\frac{\exp^{-({t}-10)}}{\left( 1+\exp^{-({t}-10)} \right)^2}{.}
	\end{align*}
	This derivative satisfies $C'(t) = C'(20-t)$ for all $t \in [0,10]$ and is strictly decreasing for $t \geq 10$. In particular, this implies that there exists no optimal upgrade policy where $T_i \in (20,30)$ for some $i=1,\ldots,N+1$. In view of~\eqref{eq:OptimalCandidatesSShape}, we observe that for any $t_2=20-t_1$ with $t_1 \in [0,10)$ and $N_2=1$,
	
	\begin{align*}
		N_1 t_1+N_2 t_2 = H \hspace{1cm} \Rightarrow \hspace{1cm} t_1 = \frac{10}{N_1-1}=\frac{10}{N-1}.
	\end{align*}
	Therefore, the possible candidates for the optimal upgrade policy are $\mathcal{T}^0$ and
	\begin{align*}
		\mathcal{T}^{N_1,0} = \left\{ \left(\frac{30}{N+1} , 0\right) \right\}, \;\;\; N \in \mathbb{N}_{\geq 2} \hspace{1.5cm} \mathcal{T}^{N_1,1} = \left\{ \left(\frac{10}{N-1} , 20 - \frac{10}{N-1} \right) \right\}, \;\;\; N \in \mathbb{N}_{\geq 2}.
	\end{align*}	
	
	We point out that any policy $\Pi \in \mathcal{T}^{N_1,1} $ leads to total costs that satisfy $\mathcal{C}^{\Pi}(H) \geq 0 > C(H)$ for every $c_0\geq1$, i.e.{,} every policy $\Pi \in \mathcal{T}^{N_1,1} $ leads to a higher total costs than when upgrades would never be executed and hence the optimal upgrade policy will never be contained in the set $\mathcal{T}^{N_1,1} $. Therefore, depending on the value of $c_0$, the optimal policy is to either never upgrade, or to upgrade $N\geq 2$ times with equidistant upgrade times that equals the remaining lifetime (i.e.{,} the time between the final upgrade and the end of the lifetime).
\end{example}

\begin{figure}[t]
	\centering
	\includegraphics[scale=0.55]{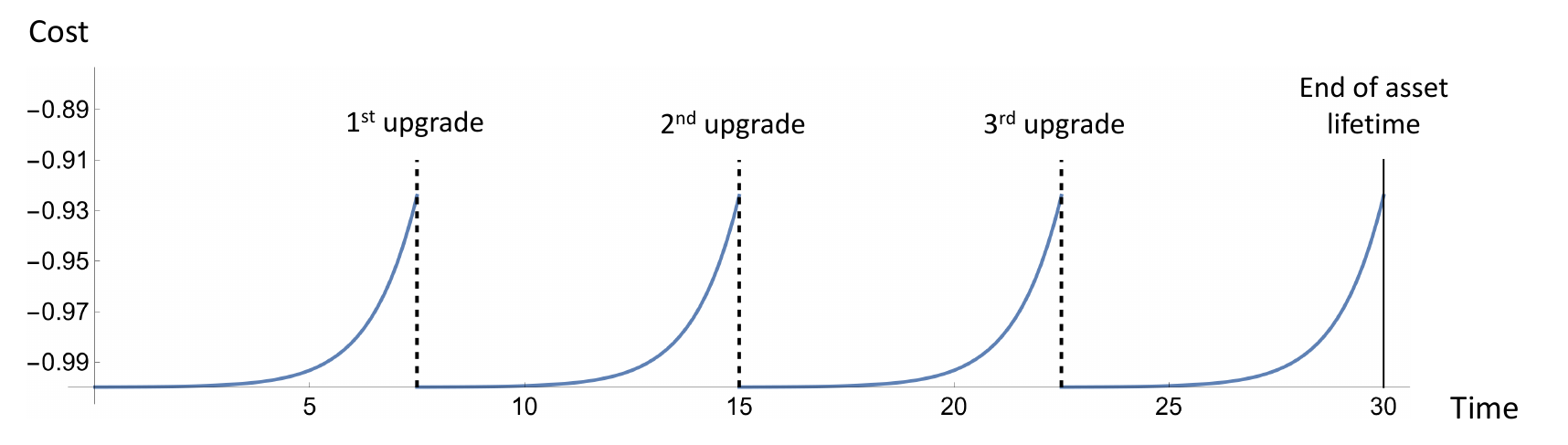}
	\caption{$N=3$ upgrades at every $7.5$ time units. Total cost is $3 \, c_0 + 4 \, C(7.5)= -0.697$ for $c_0=1$.}
	\label{fig:N3}
\end{figure}


{As shown in Example~\ref{exmS}, equidistant upgrade times can be optimal under an S-shaped cycle cost function. In Figure~\ref{fig:N3}, we illustrate the structure of the policy that upgrades three times (i.e., at every 7.5 time units)  for the instance introduced in Example~\ref{exmS}. Note that the costs of the corresponding policy can be easily calculated, and this can be repeated for other $N$ values to identify the optimal upgrade policy. Section~\ref{sec: solappbasecase} will introduce an efficient solution algorithm to identify the optimal number of upgrades $N^*$ by exploiting the analytical results established so far.}

To illustrate that it is possible to have a (unique) upgrade policy where the remaining time between the last upgrade and the end of the lifetime differs from the (other) inter-upgrade times, we provide the following example.

\begin{example}\label{ex:NotEquidistantBase}
	Consider a setting where $H=10$ years, and the system is not subjective to failure, i.e.{,} $h(t)=0$ for all $t\geq0$. Initially, the salvage value is given by $v(0)=0.15$. After $s=4.9$, an upgrade comes to the market against upgrade costs $c_0=0.75$. This event causes the missing functionality penalty to jump to $c_f(t)=0.15$ for all $t\geq 4.9$ (while $c_f(t)=0$ for all $t<4.9$). Moreover, this event leads to the salvage value to decrease rapidly till time $t=5$, after which reselling the old system yields no return. More specifically, the salvage value is given by
	\begin{align}
		v(t) = \left\{ \begin{array}{ll}
			0.15 & \textrm{if } t \leq 4.9, \\
			0.15 - 30 (t-4.9)^2 & \textrm{if } 4.9 \leq t \leq 4.95, \\
			30 (5-t)^2 & \textrm{if } 4.95 \leq t \leq 5, \\
			0 & \textrm{if } t \geq 5.
		\end{array}\right.
	\end{align}
	This leads to an S-shaped continuous cycle cost function, given by
	\begin{align}
		C(t) = \left\{ \begin{array}{ll}
			-0.15 & \textrm{if } t \leq 4.9, \\
			-0.15 + 30 (t-4.9)^2 +  0.15 (t-4.9) & \textrm{if } 4.9 \leq t \leq 4.95, \\
			-30 (5-t)^2 +  0.15 (t-4.9)& \textrm{if } 4.95 \leq t \leq 5, \\
			0.15 (t-4.9) & \textrm{if } t \geq 5.
		\end{array}\right.
	\end{align}
	The point of inflection is thus given by $x=4.95$. {This cycle cost function is illustrated in Figure~\ref{fig:cycle cost example}.}
	
	\begin{figure}[h]
		\centering
		\includegraphics[scale=0.39]{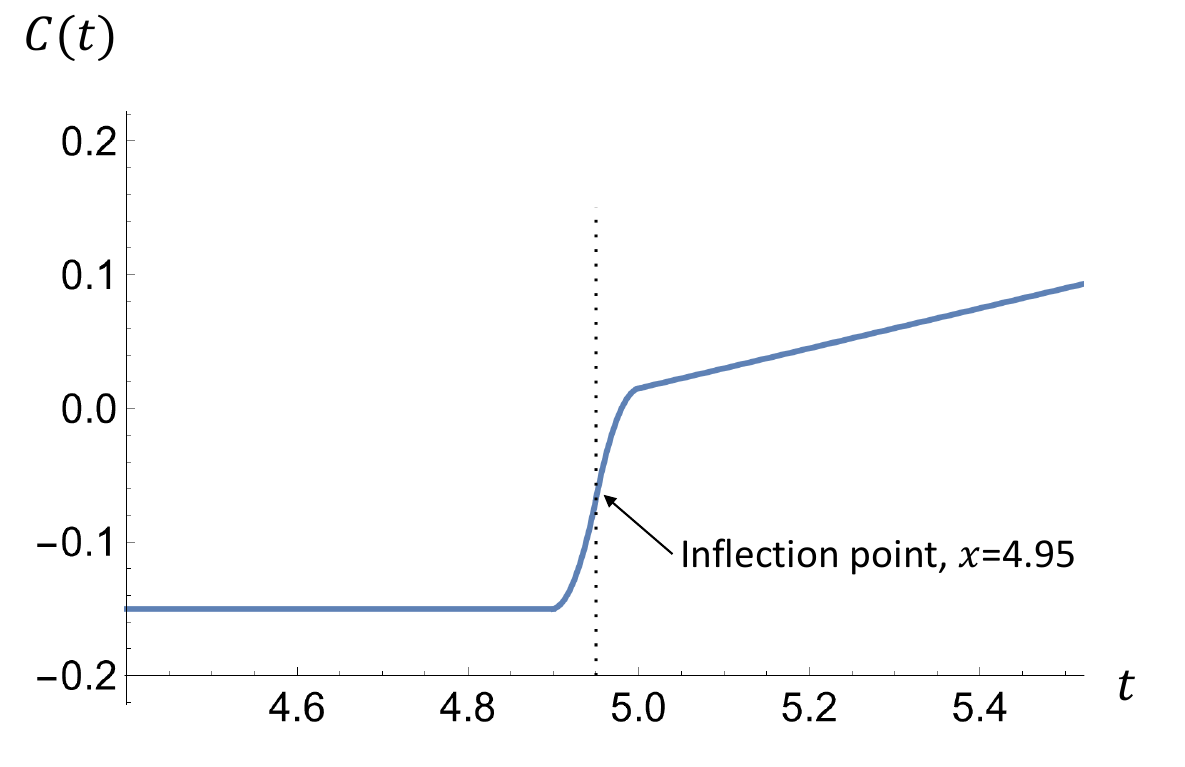}
		\caption{{Cycle cost function in Example~\ref{ex:NotEquidistantBase}}}
		\label{fig:cycle cost example}
	\end{figure}
	To determine the optimal upgrade policy, note that $\bar{N}=(C(H)+v(0))/(c_0-v(0)) = 1.525$ and hence it is never optimal to upgrade more than once. Since $C'(t)=0.15$ for all $t\geq 5$ and $C'(t) > 0.15$ for all $t \in (4.9,5)$, it is never optimal to upgrade at a time in the interval $(4.9,5)$ by~\eqref{eq:NecessaryConditionC}. Moreover, since the failure rate equals zero, there is no advantage to upgrading before $t=4.9$. In view of Proposition~\ref{prop:OptimalUpgradePolicyBase}, the optimal policy is therefore either to never upgrade, or to upgrade once at time $T_1=4.9$ or at time $T_1=5$, with corresponding total costs $C(10)=0.765$, $c_0+C(4.9)+C(5.1)=0.63$ and $c_0+2C(5)=0.78$, respectively. {By comparing these total costs, it is clear that it is optimal to upgrade at time $T_1=4.9$ years, and the remaining time between this upgrade and the end of the lifetime is $5.1$ years. This optimal upgrade policy is illustrated in Figure~\ref{fig:opt policy example}}.
\end{example}

\vspace{-0.5cm}

\begin{figure}[htb]
	\centering
	\includegraphics[scale=0.37]{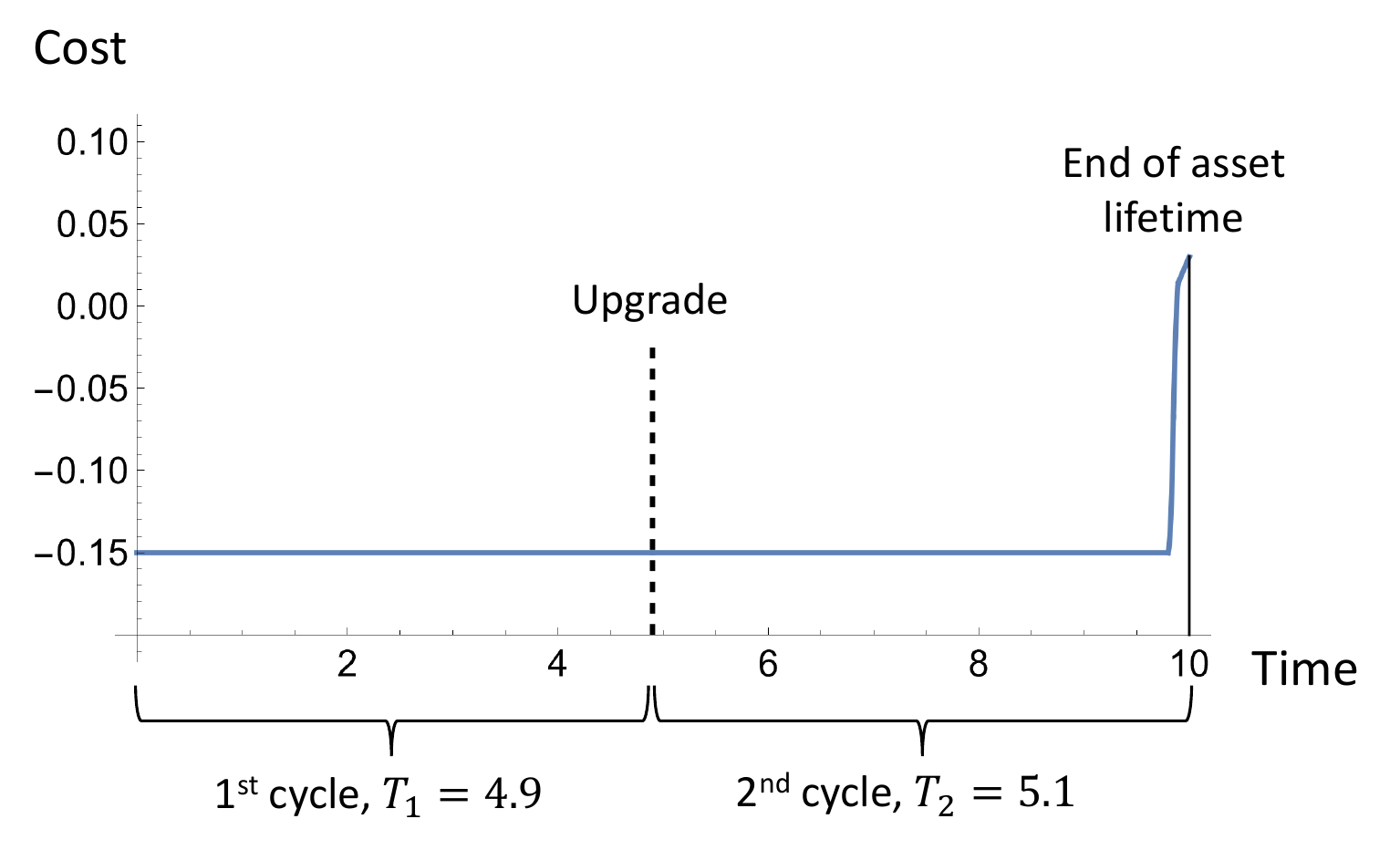}
	\caption{{Optimal upgrade policy in Example~\ref{ex:NotEquidistantBase}}}
	\label{fig:opt policy example}
\end{figure}

\vspace{-0.25cm}

\subsection{Generalized cycle costs}\label{sec:GeneralCycleCosts}
In previous sections, we established the structure of the optimal upgrade policy when the cycle cost function (which must be non-decreasing) is convex, concave or S-shaped. The question is what if the cycle cost function {does not fall in this class of functions}. This can happen in practice, and we address this generalized case in this section. To be specific, a necessary condition for optimality is given by~\eqref{eq:NecessaryConditionC}, and this condition (which assumes the differentiability of the cycle cost $C(t)$ in $t \in [0,H]$) can be exploited to generalize the results of the previous sections. More specifically, one can divide the lifetime of the system $[0,H]$ into intervals for which the cycle costs are convex and concave. That is, suppose that $[0,H]=\mathcal{H}_1 {\cup} \ldots {\cup} \mathcal{H}_k$ for some $k \in \mathbb{N}$ where $k$ is finite, and where the cycle costs are convex on intervals $\mathcal{H}_l$ if $l$ is odd, and concave on intervals $\mathcal{H}_l$ if $l$ is even, and $\mathcal{H}_i \cap \mathcal{H}_j = \emptyset$ if $i \neq j$. In Figure~\ref{fig:general_cycle_cost}, we illustrate such a partition of a generalized cycle cost function. Note that these intervals can be constructed even in the extreme case when the cycle cost is convex by choosing $k=1$ with $\mathcal{H}_1=[0,H]$, and when the cycle cost is concave by choosing  $k=2$ with $\mathcal{H}_1=\{0\}$, $\mathcal{H}_2=(0,H]$. 

\begin{figure}[h!]
	\centering
	\includegraphics[scale=0.5]{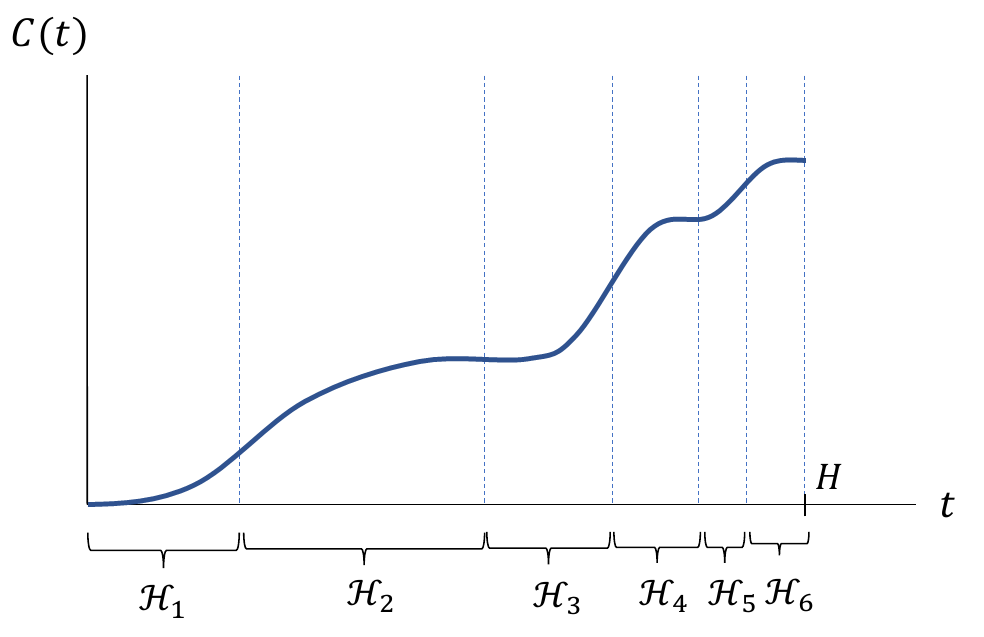}
	\caption{{Illustration of a generalized cycle cost function}}
	\label{fig:general_cycle_cost}
\end{figure}

Our next result requires the technical condition that the proposed partition of the cycle cost function can be done in finite $k$. We point out that $k$ is typically finite for instances we observe in practice, i.e.{,} the cycle costs do not behave erratically by changing from convex to concave infinitely often on the finite interval $[0,H]$. This is an intuitive observation: recall that we showed the cycle cost function $C(\cdot)$ is an aggregation of a convex non-decreasing function, and a non-decreasing function $-v(T), t\geq0$. Thus, it can be argued that the cycle cost do not behave erratically for practically reasonable forms of salvage value function $v(\cdot)$, leading to relatively small $k$ values. Under the proposed partition of the cycle cost function, we next provide a characterization of the optimal upgrade policy.
\begin{proposition}
	\label{prop: general_cycle_cost}
	Suppose that the cycle cost function can be partitioned as proposed in $k$ intervals with $k<\infty$. Then there is always an optimal upgrade policy that satisfies the following properties:
	\begin{itemize}
		\item[(i)] For every even $l$, there is at most one $i \leq N^*$ for which $T_i^* \in \mathcal{H}_l$;
		\item[(ii)] If $T_i^*, T_j^* \in \mathcal{H}_l$ for some odd $l$, then $T_i^*=T_j^*$.
	\end{itemize}
\end{proposition}

Note that {Proposition~\ref{prop: general_cycle_cost}} is a generalization of the results in the previous sections. In case that the cycle costs are convex on $[0,H]$, then Property~(ii) implies equidistant upgrade times. In case that the cycle costs are concave on $[0,H]$, Property~(i) implies that there can only be a single $T_1^* \in [0,H]$, and hence it is optimal to not upgrade at all. Finally, if the cycle costs have an S-shape, then Properties~(i) and~(ii) imply Proposition~\ref{prop:OptimalUpgradePolicyBase}.

We note that the calculation of the optimal upgrade policy may become cumbersome when the cycle costs behaves erratically, changing from convex to concave very often. As long as $k$ is {not} extremely large, the optimal inter-upgrade times can be derived relatively efficiently for any given number of upgrades, leading to a solution approach to obtain the optimal upgrade policy for a generalized cycle cost function as further discussed in  Appendix~\ref{sec:SolutionApproachB1}.

\section{Solution approach}\label{sec:SolutionApproach}
The results of the previous section can be used {to determine} the optimal policy. Note that we identified two extreme cases with respect to the penalty $c_d$, namely the base case where $c_d=0$ and an upgrade does not affect the asset's operation, or the case where $c_d=\infty$ and upgrades occur only during overhauls. We discuss the solution approach for both extreme cases first, after which we explain how one can combine the approaches to find the optimal policy for any value of $c_d$. 

\subsection{Solution approach for the base case}
\label{sec: solappbasecase}
As argued in Section~\ref{sec:GeneralCycleCosts}, the optimal inter-upgrade times can be derived efficiently for any finite values of $N$ as long as the cycle cost functions are relatively well-behaved, i.e.{,} do not change between convex and concave too often up to lifetime~$H$. In particular, we considered three special cases in Section~\ref{sec:Analysis}. If the cycle costs are concave, then the optimal policy is to never upgrade. If the cycle costs are convex, the optimal number of upgrades is the first integer at which $\Delta C^N(H) = C^{N+1}(H)-C^N(H)$ is non-negative, denoted by $N^*$. The optimal policy is then to upgrade $N^*$~times after every $H/(N^*+1)$ time units. This yields Algorithm~\ref{alg:FindOptimalAmongEquidistantConvex} as given in Appendix~\ref{sec:SolutionApproachB1}. Implicitly, this solution approach is used for Example~\ref{ex:MainExample} in Section~\ref{sec:ConvexCycleCostsBase}.

If the cycle costs are S-shaped with point of inflection~$x \in (0,H)$, we would need to check more possibilities as there may be an optimal policy where the time between the last upgrade and the asset's lifetime differs from the other inter-upgrade times. That is, we need to check the total costs under the policy of never upgrading, under the class of policies with equidistant upgrade times, and finally under the class of policies where the remaining lifetime differs from the other inter-upgrade times. No upgrading leads to total costs~$C(H)$. Next, Proposition~\ref{prop:OptimalUpgradePolicyBase} shows that if it is optimal to upgrade at least once, then there exists a policy with optimal inter-upgrade times $T_1=\ldots=T_N \leq x$. In the class of equidistant upgrade times, it therefore suffices to look at $N \geq H/x - 1$. In that case, $C(t)$ is convex in $t = H/(N+1)$, and hence Lemma~\ref{lem:ConvexityOfCN} follows through, i.e.{,}~$\mathcal{C}^N(H)$ is convex in $N \geq H/x - 1$. This implies that we can adopt a similar approach as for convex cycle cost function. That is, we determine the first integer $N\geq H/x - 1$  for which $\Delta C^N(H) = C^{N+1}(H)-C^N(H)$ is non-negative. Finally, for the class of policies with a different remaining lifetime, note that due to the S-shape of the cycle costs and Proposition~\ref{prop:OptimalUpgradePolicyBase}, it holds that an optimal policy satisfies $T_{N+1} \in (x,H)$ and $T_i = (H-T_{N+1})/N$ for all $i=1,\ldots,N$ if $N\geq 1$. Define
\begin{align}\label{eq:TotalCostsBaseSShape}
	\tilde{\mathcal{C}}^N(H,t) :=
	N c_0 + N C\left(\frac{H-t}{N}\right) + C(t), \hspace{1cm} N\geq 1,
\end{align}
the total costs where $N$ upgrades are executed at equidistant time intervals of length $(H-t)/N$, and the remaining lifetime after the last upgrade is given by $t$. By deriving $t\geq \max\{H/(N+1),x\}$ that minimizes $\tilde{\mathcal{C}}^N(H,t)$, we inherently determine the optimal policy that minimizes the total costs for this class of policies with $N$ upgrades. Checking for a finite number of possible upgrades (see Lemma~\ref{lem:FiniteNumberOfReplacements}), leads to all potential optimal upgrade policies with a different remaining lifetime. We formalize this solution approach in Algorithm~\ref{alg:FindOptimalBaseSShaped}, which is given in Appendix~\ref{sec:SolutionApproachB1}.

\begin{example}
	Recall Example~\ref{ex:NotEquidistantBase}. The derived solution could also be found by executing Algorithm~\ref{alg:FindOptimalBaseSShaped} in Appendix~\ref{sec:SolutionApproachB1}.
	\begin{enumerate}
		\item 	The first step is to check to class of equidistant inter-upgrade times. Note that $\lceil H/x -1 \rceil = 2$ and $\mathcal{C}^{3}(H) = 1.65 > 1.05 = \mathcal{C}^{2}(H)$, this part terminates after a single step and returns $N^*=2$, $T^*=5$ and $C=1.05$.
		\item The next step shows that $C(10) = 0.765 < 1.05$, and hence we update our values to $N^*=0$, $T^*=0$ and $C=0.765$. 
		\item Note that $\lfloor\bar{N}\rfloor =  \lfloor (C(H)+v(0))/(c_0-v(0)) \rfloor = 1$, and hence the for loop consists only of a single step. To determine the $t^* \in [5,5.1]$ that minimizes $\tilde{\mathcal{C}}^1(H,t)$, we observe
		\begin{align*}
			\frac{\partial \tilde{\mathcal{C}}^1(H,t)}{\partial t}=C'(H-t)+C'(t) = \left\{ \begin{array}{ll}
				60 t-299.7 & \textrm{if } 5 \leq t \leq 5.05, \\
				-60t+306.3 & \textrm{if } 5.05 \leq t < 5.1, \\
				0.15 & \textrm{if }  t = 5.1, 
			\end{array}\right.
		\end{align*}
		This function is minimized at $t^*=5.1$, for which $\tilde{\mathcal{C}}^1(H,t^*)=0.63 < 0.765 = C$. Therefore, we update our values to $C = 0.63$, $N^*=1$, and $T^*=4.9$. 
		\item Finally, the algorithm returns the optimal policy $N^*=1$, $T_1^*=4.9$, $T_2^*=5.1$ and corresponding minimum costs $\mathcal{C}^*(H)=0.63$. 
	\end{enumerate}
	
\end{example}

When the cycle costs are not convex, concave or S-shaped, a more involved analysis needs to be performed. As indicated in Section~\ref{sec:GeneralCycleCosts}, this can still be done as long as the number of times that the cycle costs changes from convex to concave or vice versa is not too large, which is the case for most practical instances. That is, we need to find an optimal policy under the class of policies that satisfy the two properties indicated in Proposition~\ref{prop: general_cycle_cost}. By exploiting these two properties, a solution approach can be designed in a similar way as we did for the S-shaped function, see Appendix~\ref{sec:SolutionApproachB1}.

\subsection{Solution approach when only upgrading during overhauls}\label{sec:SolutionAppOnlyOverhauls}
The extreme case $c_d=\infty$ where upgrades are only executed during overhauls reduces the number of possible upgrade plans significantly. During the asset's lifetime, one needs to decide only at the $m$ overhauls whether to upgrade or not, where $m$ is finite and relatively small in practice. This can be solved via dynamic programming. In essence, we condition on the first occurrence of executing an upgrade at an overhaul, if any. Suppose that this occurs at some time $T \in (0,H)$, then we split the horizon~$H$ in two parts $[0,T]$ and $(T,H]$. The total costs during the first part is given by $c_0+C(T)$, while the total costs poses an identical problem except that the time horizon has changed to a smaller value $H-T$. Since $m$ is relatively small, we can efficiently work backwards in time to determine the optimal policy. Details of this approach are given in Algorithm~\ref{alg:FindOptimalNoIntermediate} of Appendix~\ref{sec:SolutionApproachB2}.

\begin{example}
	Reconsider the example as described in Example~\ref{ex:MainExample}. For the additional parameters, suppose that $c_d=\infty$ and $M_i=5$ for $i=1,\ldots,6$ in both settings, i.e.{,} there is an overhaul every 5~years. In setting~A, recall that the optimal policy in the base case was to upgrade after 15 years. Since this coincides with an overhaul, this policy remains the optimal choice and consequently, the corresponding minimal total costs of $27.308$ are not affected. For setting~B, the optimal policy does change with respect to the base model: instead of upgrading every six years, we upgrade every five years together with a planned overhaul. This leads to total costs of $38.732$. 
\end{example}

\subsection{On the optimal solution for the general case}
\label{subsec: solngeneral}
To find the optimal policy for the general case, a similar dynamic programming approach can be taken. To determine the optimal policy, we condition on the first occurrence we execute an upgrade at an overhaul (if any). This cuts the lifetime in two pieces that can be optimized separately. The expected total costs in the first piece can be determined using a strategy similar to the base model, since we conditioned that no upgrade is executed during an overhaul in this interval. The second piece of the interval is an identical problem for the general model, but with a shorter horizon length. Using dynamic programming, we can then efficiently determine the optimal policy.

More specifically, suppose that the first time an upgrade is executed at an overhaul is at time $T = \sum_{i=1}^j M_i$ for some $j \leq m$. In the first interval $[0,T)$, no upgrade is executed at an overhaul moment. This corresponds to the base case, with a slight modification that the price of executing an upgrade comes with an additional penalty $c_d>0$. That is, for a given policy~$\Pi$, let $\hat{\mathcal{C}}^\Pi(T)$ denote the total costs in the base case with price $c_0+c_d$ for an upgrade, horizon length~$T$, number of upgrades $\hat{N}$ with corresponding inter-upgrade times $T_1,\ldots, T_{\hat{N}}$ and $T_{\hat{N}+1}=T-\sum_{i=1}^{\hat{N}} T_i$. Equivalently to~\eqref{eq:TotalCostsBaseCase}, the total costs are given by
\begin{align}\label{eq:TotalCostsBaseWithPenaltyCosts}
	\hat{\mathcal{C}}^\Pi(T) = \hat{N} \cdot (c_0+c_d) + \sum_{i=1}^{\hat{N}+1} C(T_i).
\end{align}

The price of the upgrade at time $T$ itself is given by $c_0$. Finally, we need to consider the total costs for the second interval. However, this poses an identical problem with a shorter horizon. Since $m$ is relatively small, we can therefore use dynamic program efficiently to determine the optimal policy, see Algorithm~\ref{alg:FindOptimalGeneral} in Appendix~\ref{sec:SolutionApproachB1}.

\begin{example}
	Reconsider the example as described in Example~\ref{ex:MainExample}, together with $c_d=1.5$ and $M_i=10$ for $i=1,2,3$ in both settings. By executing Algorithm~\ref{alg:FindOptimalGeneral} in Appendix~\ref{app:SolutionApproach}, we determine the optimal policy which is given as follows. In case~A, it is optimal to execute an upgrade twice after every 10~years, i.e.{,} together with the planned overhauls. The associated costs are $28.027$. In setting~B, it turns out that it is optimal to upgrade once after 10~years during an overhaul. In the remaining 20 years, it is optimal to upgrade two more times at times that do not coincide with overhauls. That is, the optimal policy prescribes to upgrade at times $\{10,16\frac{2}{3}, 23\frac{1}{3}\}$ years with associated total costs of $41.794$. 
\end{example}


\section{Effects of the input parameter on the optimal total costs}\label{sec:Sensitivity}
We observed that in the base case, the behavior of the cycle cost function determines the structure of the optimal upgrade policy. When a penalty is incurred for upgrading while there is no overhaul, the optimal upgrade policy can be determined using dynamic programming (that uses the structural results from the base case). The final outcome depends heavily on the values and functions of the different input parameters. In this section, we take a closer look at this notion both analytically and numerically. We note that computational times of our solution approach is low in all instances in the examples below. In particular, these times are not influenced by the value of $c_d$. The proofs of the analytical results can be found in Appendix~\ref{app:sensitivityProofs}.

\subsection{The effect of the penalty $c_d$ for upgrades outside overhaul moments}
In previous sections, we already observed that there is a strong dependence on the penalty~$c_d$. Whenever this penalty is very small, the optimal policy is the same as in the base case. {To be specific,} if the base case prescribes to upgrade only at moments that happen to not coincide with the overhaul moments, then this {also} holds for all sufficiently small penalty values of $c_d$. If there is a relatively high penalty $c_d$, it would be very costly to execute upgrades outside overhaul moments, and hence it would not be optimal to do so. In fact, whenever it is optimal to only execute upgrades during overhauls for a given penalty $\bar{c}_d$, then the same strategy is optimal for every case with $c_d \geq \bar{c}_d$. This is rather intuitive: higher penalties strictly increase the total costs for any policy that executes upgrades outside the overhauls, while policies that only upgrade during overhauls yield the same total costs. 
{Proposition~\ref{prop:MainCdResult} generalizes this notion, and shows that,} as the penalty $c_d$ increases, the optimal policy can only change to a policy where the number of upgrades that are not jointly executed with overhauls is smaller.

\begin{proposition}\label{prop:MainCdResult}
	As the penalty $c_d$ increases, the optimal policy can only change to one where the number of upgrades outside overhauls is {reduced}.
\end{proposition}

In particular, if a policy $\Pi^*$ prescribes to only upgrade during overhauls, then $S^{\Pi^*}=0$. If $\Pi^*$ is the optimal policy for some penalty value $\bar{c}_d$, then Proposition~\ref{prop:MainCdResult} implies that the optimal policy also only upgrades during overhauls for all $c_d \geq \bar{c}_d$. 

\begin{corollary}\label{cor:StayWithOverhaulMomentsCd}
	Suppose that for a penalty $\bar{c}_d$, an optimal policy $\Pi^*$ prescribes to upgrade only at planned overhauls. For any setting with $c_d \geq  \bar{c}_d$ and where the other parameter settings are the same, $\Pi^*$ is also an optimal policy.
\end{corollary}

We would like to point out that Corollary~\ref{cor:StayWithOverhaulMomentsCd} implies scenarios where the optimal upgrade solution is insensitive to the value of~$c_d$. Whenever the base case leads to an upgrade plan where the timing of the upgrades coincide with the overhauls, Corollary~\ref{cor:StayWithOverhaulMomentsCd} implies that this is also the optimal strategy for all $c_d>0$. In particular, if it is optimal to never upgrade in the base case, then this is also the optimal strategy for positive values of penalty~$c_d$. 

\begin{example}
	\label{exm:exm7}
	To illustrate Proposition~\ref{prop:MainCdResult}, we reconsider setting~B of Example~\ref{ex:MainExample}. That is, suppose that the horizon length is given by $H=30$, the price is $c_0=4$, and we have cycle costs
	\begin{align*}
		C(t) = \frac{t}{3} + \frac{3}{16} \left(\frac{t}{3}\right)^3 + \frac{1}{10}
		t^{1.1}.
	\end{align*}

	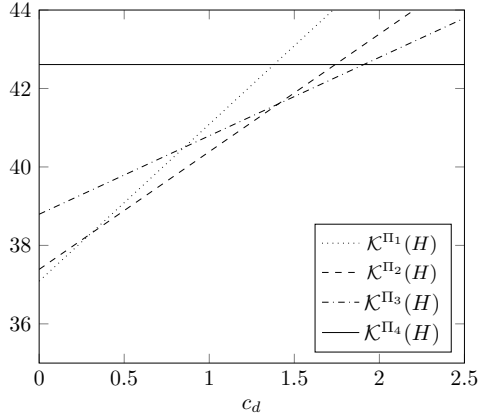
\begin{figure}[htb!]
		\centering
		\caption{Total costs $\mathcal{K}^{\Pi}(H)$ under policies $\Pi_1=\{6,12,18,24\}$, $\Pi_2=\{7.5,15,22.5\}$, $\Pi_3=\{10,16\frac{2}{3}, 23\frac{1}{3}\}$ and $\Pi_4=\{10,20\}$.}
		\label{fig:EffectOfCd}
		\begin{tikzpicture}[scale=0.825]
			\begin{axis}[domain=0:30, xmin=0, xmax=2.5,ymin=35,ymax=44,samples=300,legend style={at={(1.1,1)}, anchor=north west},legend pos=south east,no marks,legend entries={\small $\mathcal{K}^{\Pi_1}(H)$,\small $\mathcal{K}^{\Pi_2}(H)$,$\mathcal{K}^{\Pi_3}(H)$,$\mathcal{K}^{\Pi_4}(H)$},xlabel=$c_d$]

				\addplot[dotted] {37.0887+4*\x};
				\addplot[dashed] {37.3884+3*\x};
				\addplot[dashdotted] {38.794+2*\x};
				\addplot[] {42.6101};
				
			\end{axis}
			
		\end{tikzpicture}
	\end{figure}
	
	Moreover, two overhauls are planned every 10 years, i.e.{,} $M_i=10$, $i=1,2,3$. We consider how the optimal upgrade policy is influenced by $c_d$. In view of Figure~\ref{fig:EffectOfCd}, the optimal {upgrade policy} is:
	\begin{itemize}
		\item If $c_d \in [0,0.29973)$, upgrade every 6 years at times $\{6,12,18,24\}$ with $\mathcal{K}^*(H) = 16+5 C(6)+4 c_d$;
		\item If $c_d \in [0.29973,1.40559)$, upgrade at times $\{7.5,15,22.5\}$ with $\mathcal{K}^*(H) = 12+4C(7.5)+3 c_d$;
		\item If $c_d \in [1.40559,1.90805)$, upgrade during one overhaul moment and two times after $6\frac{2}{3}$ years, i.e.{,} at times $\{10,16\frac{2}{3}, 23\frac{1}{3}\}$ or $\{6\frac{2}{3}, 13\frac{1}{3}, 20\}$ with $\mathcal{K}^*(H) = 12+C(10)+3C(6\frac{2}{3})+2 c_d$;
		\item If $c_d \geq 1.90805$, upgrade twice during the overhauls at times $\{10,20\}$ with $\mathcal{K}^*(H) = 8+3C(10)$.
	\end{itemize} 
\end{example}
{In Example~\ref{exm:exm7}, we observe that, as the penalty of upgrading outside an overhaul moment ($c_d$) is increasing, the number of upgrades outside an overhaul moment is non-increasing (as already shown in Proposition~\ref{prop:MainCdResult}), and the number of upgrades at an overhaul is non-decreasing. Also, we see that the total number of upgrades is a non-increasing function of $c_d$. This is intuitive as one can expect that, given everything else is equal, the higher a cost associated with an upgrade, the lower the number of upgrades. However, interestingly, this does not hold in general as illustrated next.
	\begin{example}
		\label{exm:exm8}
		Suppose that everything is the same as in Example~\ref{exm:exm7}, except that the cycle cost function is now from the setting A of Example~\ref{ex:MainExample}, i.e., 
		\begin{align*}
			C(t) = \frac{t}{3} + \frac{3}{16} \left(\frac{t}{3}\right)^2 + \frac{1}{10}
			t^{1.1}.
		\end{align*}
		The optimal upgrade policy is then as follows:
		\begin{itemize}
			\item If $c_d \in [0,0.71872)$, upgrade only once at time $\{15\}$ with $\mathcal{K}^*(H) = 4+2 C(15)+ c_d$;
			\item If $c_d \geq 0.71872$, upgrade at times $\{10,20\}$ with $\mathcal{K}^*(H) = 8 + 3C(10)$;
		\end{itemize}
	\end{example}
	Example~\ref{exm:exm8} shows that, as $c_d$ increases, the scenarios where upgrading occurs during overhauls may become relatively more attractive, leading to more frequent upgrades. 
}

\subsection{The effect of the number of overhaul moments}
As the number of overhauls increases, it generates more possibilities to execute upgrades without having to pay an additional penalty $c_d$. This does not necessary imply that the number of upgrades (during overhauls) increases together with the number of overhauls. The lack of this monotonicity property can be easily explained as follows. There is an optimal upgrade policy in the base case, i.e., the optimal policy to minimize costs if no penalty $c_d$ is incurred. This is attained in the general setting if the overhauls coincide exactly with the upgrade moments in this policy. Therefore, one cannot expect to do better than this, even if more overhauls are planned. To illustrate this notion, we consider the following example.

\begin{example}
	Reconsider setting~B of Example~\ref{ex:MainExample} with $H=30$, $c_0=4$ and $c_d=5$. In the base case when no penalty is incurred, the optimal policy is to upgrade $N=4$ times (every 6 years) with minimal costs of 37.0887. This a lower bound on the total costs for the general setting. If the overhauls are planned equidistantly, this implicitly also implies that this optimal solution is obtained whenever the number of overhaul moments is given by $m=5k-1$ for some $k \in \mathbb{N}$. Indeed, in Table~\ref{tab:EffectNumberOverhaulMoments} we observe that the total costs are higher when e.g.{,}~$m=5$ than when $m=4$. {Moreover, note that as the number of overhauls grows from $m=1$ to $m=2$, the total number of upgrades decreases.}
	
	\begin{table}[h!]
		\centering
		\caption{Optimal upgrade policy and minimal costs in case of overhaul moments $M_i=H/(m+1)$, $i=1,\ldots,m$. The number in bold correspond to moments that coincide with overhauls.}
		\label{tab:EffectNumberOverhaulMoments}
		\begin{tabular}{ | l | l | l |}
			\hline
			$m$ & Optimal upgrade policy $\Pi^*$ & $K^*(H)$ \\ 
			\hline
			0 & \{7.5,15,22.5\} & 52.3884  \\
			1 & \{7.5,\textbf{15},22.5\} & 47.3884 \\
			2 & \{\textbf{10},\textbf{20}\} & 42.6101  \\
			3 & \{\textbf{7.5},\textbf{15},\textbf{22.5}\} & 37.3884 \\
			4 & \{\textbf{6},\textbf{12},\textbf{18},\textbf{24}\} & 37.0887 \\
			5 & \{\textbf{5},\textbf{10},\textbf{15},\textbf{20},\textbf{25}\} & 38.7322 \\
			\hline
		\end{tabular}
		
	\end{table}
\end{example}

\subsection{The effect of price $c_0$}
Naturally, as the price $c_0$ increases, we can expect fewer upgrades and indeed, this intuitive notion turns out to be correct. Note that Assumption~\ref{ass:UpgradeCosts} states that we only consider values of $c_0$ such that $c_0>v(0)$.

\begin{lemma}
	The optimal number of upgrades $N^*$ is non-increasing in $c_0$.
	\label{lem:NDecreasingInc0}
\end{lemma}

Next, we provide an example to illustrate this result.

\begin{example}
	Reconsider setting~B of Example~\ref{ex:MainExample} with $H=30$, $c_d=5$ and two overhaul moments after every 10 years. In Table~\ref{tab:EffectOfC0} we display the total costs under the optimal upgrade policy under the condition that the number of upgrades~$N$ is fixed. We point out that if $c_0=0$, then the optimal policy is to upgrade after every 5 years with $N^*=5$, and hence we do not consider larger values of~$N$ in Table~\ref{tab:EffectOfC0}. We conclude that 	
	\begin{align*}
		N^* = \left\{ \begin{array}{ll}
			5 & \textrm{if } c_0 \in [0,0.2926] \\
			{2} & {\textrm{if } c_0 \in [0.2926,10.5944]} \\
			1 & \textrm{if } c_0 \in [10.5944,135.907] \\
			0 & \textrm{if } c_0 > 135.907,
		\end{array}\right.
	\end{align*}
	where the corresponding optimal policies can be found in Table~\ref{tab:EffectOfC0}. 
	\begin{table}[htb!]
		\centering
		\caption{Value of $\mathcal{K}^{\Pi}(H)$  under optimal upgrade policy~$\Pi$ conditioned on having $N$ upgrades.} 
		\label{tab:EffectOfC0}
		\begin{tabular}{ | l | l | l |}
			\hline
			$N$ & $\Pi$ & $\mathcal{K}^\Pi(H)$ \\ 
			\hline
			0 & \{$\emptyset$\} & 201.715 \\
			1 & \{15\} & 65.8081$ + c_0$ \\
			2 & \{10,20\} & 34.6101$ + 2 c_0$ \\
			5 & \{5,10,15,20,25\} & 33.7322$ + 5 c_0$ \\
			\hline
		\end{tabular}
	\end{table}
	
	Although the optimal number of upgrades changes as $c_0$ increases, we observe that there is a rapid change from $N^*=5$ upgrades to $N^*=2$ in a relatively short interval ($c_0 \in [0.29,0.30]$). 
		For the values of $c_0$ within this interval, we also note that it matters relatively little which policy (i.e., $N \in \{2,3,4,5 \})$ is chosen, as all lead to similar total costs.

	\end{example}

	\subsection{The effect of the cycle costs}
	A final aspect concerns the effect of the cycle costs, which can be seen as the combined contribution of the salvage value, the missing functionality penalty function and the failure aspect. Already in the base case, we observed the strong effect of the cycle costs on the structural properties of the optimal upgrade policy. In this section, we focus on whether we can say more about the upgrade policy whenever the cycle costs start to increase rapidly. 
	
	Intuitively, we would always upgrade before a point at which the cycle costs increase very strongly. To a certain extent, this can be made rigorous through the following result. Suppose that the cycle costs increase extremely after some time $z \in (0,H)$. If it is already optimal to upgrade multiple times with all inter-upgrade times occurring before time $z$ in an alternate setting where we linearize the cycle costs after time $z$, i.e.{,} a setting where a less rapid increase in cycle costs is assumed after time $z$, then this must also be optimal for the original setting. 
	
	\begin{lemma}\label{lem:SteepIncreaseCycleCostsAfterZ}
		Consider an alternate setting with the same parameter settings, except $\tilde{C}(t) = C(z)+(t-z)C'(z)$ for all $t \geq z$ for some $z \in (0,H)$. If the optimal policy $\tilde{\Pi}^*$ for the alternate setting is to upgrade $\tilde{N}^*$ times with $\tilde{T}_i^* \leq z$ for all $i=1,\ldots,\tilde{N}^*+1$, then $\tilde{\Pi}^*$ is also the optimal upgrade policy for the original setting.
	\end{lemma}
	
	We point out that Lemma~\ref{lem:SteepIncreaseCycleCostsAfterZ} may imply more information about the structure of the optimal upgrade policy. For example, suppose that there exists an alternate setting with the properties as in Lemma~\ref{lem:SteepIncreaseCycleCostsAfterZ} for some $z \in (0,H)$, and the cycle costs $C(t)$ are convex on $[0,z]$. By construction, note that $\tilde{C}(t)$ is convex on the entire interval $[0,H]$, and hence the optimal upgrade policy $\tilde{\Pi}^*$ is one with equidistant upgrade times. In other words, there is an optimal policy where we upgrade after equidistant time intervals that are at most equal to $z$. 
	
	To conclude this section, we would like to point out that Lemma~\ref{lem:SteepIncreaseCycleCostsAfterZ} does not imply that the number of upgrades always increases as the total cycle cost curve steepens. In fact, this is actually not true as we illustrate in the following example.
	
	\begin{figure}[htb!]
		\centering
		\caption{Upgrade costs over $[0,t]$ if $t<H$.}
		\label{fig:CounterIntuitiveExample}
		\begin{tikzpicture}[scale=0.8]
			\begin{axis}[domain=0:0.5, xmin=0, xmax=0.5,ymin=0,ymax=0.6,samples=300,legend style={at={(1.1,1)}, anchor=north west},legend pos=north west,no
				marks,legend entries={\small $C_A(t)$,\small $C_B(t)$},xlabel=$t$]
				
				\addplot[dashed] {\x+\x^2/10};
				\addplot[] {2/3*\x^(3/2)};
				
			\end{axis}
			
		\end{tikzpicture}
	\end{figure}
	
	\begin{example}
		Suppose that $c_d=0$ (base case) and $c_0=0.02$, and $H=0.5$. For the cycle costs, consider the following. In setting A, let the cycle costs be given by	$C_A(t) = t + t^2/10 = \int_{0}^t (1+x/5) \, dx $,
		and for setting B, $C_B(t) = \frac{2}{3} t^{3/2} = \int_{0}^t \sqrt{x} \, dx$.
		
		We point out that $(1+x/5) \geq \sqrt{x}$ for all $x \in [0,H]$, and hence it also holds that $C_A(t) \geq C_B(t)$ for all $t \in [0,H]$. In particular, we see that the increase in cycle costs is steeper for setting~A than for setting~B, see Figure~\ref{fig:CounterIntuitiveExample}. Yet, as we see in Table~\ref{tab:CounterIntuitiveExample}, the total costs are minimized if $N^*_A = 0$ for setting~A, and if $N^*_B = 2$ in setting~B. In other words, although the total costs increase is steeper in setting~A, the optimal number of upgrades is larger in setting~B.

		\begin{table}[htb!]
			\centering
			\caption{Total costs values for setting A and B.} 
			\label{tab:CounterIntuitiveExample}
			\begin{tabular}{ | l | l | l |}
				\hline
				$N$ & $\mathcal{C}^N_A(H)$ & $\mathcal{C}^N_B(H)$\\ 
				\hline
				0 & 0.525 & 0.2357 \\
				1 & 0.5325 & 0.1867 \\
				2 & 0.5483 & 0.1761 \\
				3 & 0.5663 & 0.1779 \\
				\hline
			\end{tabular}
		\end{table}

		\end{example}

		\section{Conclusion}\label{sec:Conclusion}
		This paper considers a continuous-time stochastic model in order to determine an optimal upgrade policy for systems in an asset. This model is novel for combining several aspects that are relatively under-examined in existing literature: functionality gap, a predetermined overhaul plan, age-dependent cost functions and finite lifetime of the asset. For this model, we analytically characterize the structure of the optimal upgrade policy. For the base case {with no penalty for not upgrading at an overhaul moment}, we establish that it is optimal to never upgrade if the cycle costs are concave, and to upgrade at {equidistant time intervals} if the cycle costs are convex. {Based on these two results, we further characterize the structure of the optimal policy first for S-shaped cycle costs, and then for generalized cycle costs.} We use these {analytical} results as building blocks to design an efficient solution approach based on dynamic programming when the penalty for executing upgrades outside overhauls is non-negative.
		
		Naturally, the optimal upgrade policy depends heavily on the input parameters. Many intuitive sensitivity results can be made rigorous. For example, the optimal number of upgrades {cannot increase} as the {upgrade} price becomes larger. A more subtle approach needs to be taken for other input parameters: as the penalty {of upgrading at a non-overhaul moment} increases, the optimal upgrade policy can only change to one where the number of upgrades that are not executed jointly with overhauls is less. {However, surprisingly, that does not imply that the optimal number of upgrades is a non-increasing function of this penalty. Also, when the number of overhauls {increases}, the optimal number of upgrades can decrease.}
		Finally, the number of upgrades also does not necessarily have to be non-decreasing with the steepening of the cycle costs. 
		
		{Motivated by a real-life project involving an asset owner and system supplier in partnership, we developed a model that can be  used to collaboratively plan the system upgrades. However, for certain product categories, the incentives may not yet exist to make the upgrade decisions together. A future research direction can be to study the problem introduced in our paper from a game theoretical point of view to build a model of strategic interactions of the asset owner and the system supplier. Another future research direction can be to characterize the optimal upgrade policies that are robust against the lifetime extensions of assets.} Finally, future research can extend the problem considered in this paper to the upgrade planning of multiple systems in an asset.

\bibliographystyle{apalike}
\bibliography{UpgradePolicy.bib}

\begin{thebibliography}{}

\bibitem[IEC, 2019]{IEC62402}
 (2019).
\newblock {\em Obsolescence management - application guide}.
\newblock IEC 62402.

\bibitem[Arts et~al., 2019]{Arts2019}
Arts, J., Basten, R., and van Houtum, G.-J. (2019).
\newblock Maintenance service logistics.
\newblock In Zijm, H., Klumpp, M., Regattieri, A., and Heragu, S., editors,
  {\em Operations, Logistics and Supply Chain Management}, pages 493--517.
  Springer International Publishing, Cham.

\bibitem[Asadi et~al., 2022]{asadi2022overview}
Asadi, M., Hashemi, M., and Balakrishnan, N. (2022).
\newblock An overview of some classical models and discussion of the
  signature-based models of preventive maintenance.
\newblock {\em Applied Stochastic Models in Business \& Industry}.

\bibitem[Barlow and Hunter, 1960]{Barlow1960}
Barlow, R. and Hunter, L. (1960).
\newblock Optimum preventive maintenance policies.
\newblock {\em Operations Research}, 8(1):90--100.

\bibitem[Barlow and Proschan, 1965]{Barlow1965}
Barlow, R. and Proschan, F. (1965).
\newblock {\em Mathematical Theory of Reliability}.
\newblock John Wiley \& Sons, New York, NY.

\bibitem[Behfard et~al., 2015]{Behfard2015}
Behfard, S., van~der Heijden, M.~C., Al~Hanbali, A., and Zijm, W.~H. (2015).
\newblock Last time buy and repair decisions for spare parts.
\newblock {\em European Journal of Operational Research}, 244(2):498--510.

\bibitem[Boland, 1982]{Boland1982}
Boland, P.~J. (1982).
\newblock Periodic replacement when minimal repair costs vary with time.
\newblock {\em Naval Research Logistics Quarterly}, 29(4):541--546.

\bibitem[Boland and Proschan, 1982]{BolandProschan1982}
Boland, P.~J. and Proschan, F. (1982).
\newblock Periodic replacement with increasing minimal repair costs at failure.
\newblock {\em Operations Research}, 30(6):1183--1189.

\bibitem[Chien, 2010]{chien2010effect}
Chien, Y.-H. (2010).
\newblock The effect of a pro-rata rebate warranty on the age replacement
  policy with salvage value consideration.
\newblock {\em IEEE Transactions on Reliability}, 59(2):383--392.

\bibitem[Dagpunar and Jack, 1994]{Dagpunar1994}
Dagpunar, J. and Jack, N. (1994).
\newblock Preventative maintenance strategy for equipment under warranty.
\newblock {\em Microelectronics Reliability}, 34(6):1089--1093.

\bibitem[{de Jonge} and Scarf, 2020]{DeJonge2020}
{de Jonge}, B. and Scarf, P.~A. (2020).
\newblock A review on maintenance optimization.
\newblock {\em European Journal of Operational Research}, 285(3):805--824.

\bibitem[Gertsbakh, 2000]{gertsbakh2000reliability}
Gertsbakh, I. (2000).
\newblock {\em Reliability theory: with applications to preventive
  maintenance}.
\newblock Springer Science \& Business Media.

\bibitem[Hopp and Nair, 1994]{Hopp1994}
Hopp, W.~J. and Nair, S.~K. (1994).
\newblock Markovian deterioration and technological change.
\newblock {\em IIE Transactions}, 26(6):74--82.

\bibitem[Li and Tomlin, 2022]{li2022after}
Li, C. and Tomlin, B. (2022).
\newblock After-sales service contracting: Condition monitoring and data
  ownership.
\newblock {\em Manufacturing \& Service Operations Management},
  24(3):1494--1510.

\bibitem[Mercier, 2008]{Mercier2008}
Mercier, S. (2008).
\newblock Optimal replacement policy for obsolete components with general
  failure rates.
\newblock {\em Applied Stochastic Models in Business and Industry},
  24(3):221--235.

\bibitem[Nair, 1995]{Nair1995}
Nair, S.~K. (1995).
\newblock Modeling strategic investment decisions under sequential
  technological change.
\newblock {\em Management Science}, 41(2):282--297.

\bibitem[Nair and Hopp, 1992]{Nair1992}
Nair, S.~K. and Hopp, W.~J. (1992).
\newblock A model for equipment replacement due to technological obsolescence.
\newblock {\em European Journal of Operational Research}, 63(2):207--221.

\bibitem[Nguyen et~al., 2013]{Nguyen2013}
Nguyen, T.~K., Yeung, T.~G., and Castanier, B. (2013).
\newblock Optimal maintenance and replacement decisions under technological
  change with consideration of spare parts inventories.
\newblock {\em International Journal of Production Economics}, 143(2):472--477.

\bibitem[{\"O}ner et~al., 2015]{Oner2015}
{\"O}ner, K.~B., Kiesm{\"u}ller, G.~P., and van Houtum, G.-J. (2015).
\newblock On the upgrading policy after the redesign of a component for
  reliability improvement.
\newblock {\em European Journal of Operational Research}, 244(3):867--880.

\bibitem[Pierskalla and Voelker, 1976]{Pierskalla1976}
Pierskalla, W.~P. and Voelker, J. (1976).
\newblock A survey of maintenance models: The control and surveillance of
  deteriorating systems.
\newblock {\em Naval Research Logistics Quarterly}, 23:353--388.

\bibitem[Rajagopalan et~al., 1998]{Rajagopalan1998b}
Rajagopalan, S., Singh, M.~R., and Morton, T.~E. (1998).
\newblock Capacity expansion and replacement in growing markets with uncertain
  technological breakthroughs.
\newblock {\em Management Science}, 44(1):12--30.

\bibitem[Sanoubar et~al., 2020]{Sanoubar2020}
Sanoubar, S., Maillart, L.~M., and Prokopyev, O.~A. (2020).
\newblock Age-replacement policies under age-dependent replacement costs.
\newblock {\em IISE Transactions}, 0(0):1--12.

\bibitem[Schouten et~al., 2022]{schouten2022maintenance}
Schouten, T.~N., Dekker, R., Hekimo{\u{g}}lu, M., and Eruguz, A.~S. (2022).
\newblock Maintenance optimization for a single wind turbine component under
  time-varying costs.
\newblock {\em European Journal of Operational Research}, 300(3):979--991.

\bibitem[Segawa et~al., 1992]{Segawa1992}
Segawa, Y., {Masamitsu Ohnishi}, and {Toshihide Ibaraki} (1992).
\newblock Optimal minimal-repair and replacement problem with age dependent
  cost structure.
\newblock {\em Computers \& Mathematics with Applications}, 24(1):91--101.

\bibitem[Sols et~al., 2012]{Sols2012}
Sols, A., Romero, J.~J., and Cloutier, R.~J. (2012).
\newblock Performance-based logistics and technology refreshment programs:
  Bridging the operational-life performance capability gap in the spanish f-100
  frigates.
\newblock {\em Systems Engineering}, 15:422--432.

\bibitem[Tilquin and Cléroux, 1975]{Tilquin1975}
Tilquin, C. and Cléroux, R. (1975).
\newblock Periodic replacement with minimal repair at failure and general cost
  function.
\newblock {\em Journal of Statistical Computation and Simulation}, 4(1):63--77.

\bibitem[Tomczykowski, 2003]{Tomczykowski2003}
Tomczykowski, W. (2003).
\newblock A study on component obsolescence mitigation strategies and their
  impact on {R}\&{M}.
\newblock In {\em Annual Reliability and Maintainability Symposium}, pages
  332--338.

\bibitem[Wang, 2002]{Wang2002}
Wang, H. (2002).
\newblock A survey of maintenance policies of deteriorating systems.
\newblock {\em European Journal of Operational Research}, 139(3):469 -- 489.

\bibitem[Zhao et~al., 2017]{Zhao2017}
Zhao, X., Al-Khalifa, K.~N., {Magid Hamouda}, A., and Nakagawa, T. (2017).
\newblock Age replacement models: A summary with new perspectives and methods.
\newblock {\em Reliability Engineering \& System Safety}, 161:95--105.

\end{thebibliography}

\newpage
\appendix

\section{Notation}\label{app:Notation}
\begin{table}[h!]
	\centering
	\caption{Overview of notation.}
	\label{tab:NotationOverview}
	\begin{tabular}{ | l | p{13.5cm} |}
		\hline
		Variable & Meaning \\
		\hline
		$H$ & Lifetime of the asset \\
		$c_0$ & Price of a system upgrade\\
		$v(t)$ & Salvage value of the current system that has been in use for time $t$ \\
		$c_f(t)$ & Penalty for missing functionality of the current system that has been in use for time $t$ \\
		$h(t)$ & Failure rate of a system that has been in use for time $t$ \\
		$k(t)$ & Expected repair costs for a failed system that has been in use for time $t$ \\
		$c_d$ & Penalty for executing an upgrade not during an overhaul \\
		$m$ & The number of overhauls \\
		$M_{i}$ & Time between overhaul $i-1$ and $i$, with $M_0=0$ and $\sum_{i=1}^{m+1} M_i=H$  \\
		$\Pi$ & The upgrade policy (specified by the upgrade times) \\
		$N$ & The number of upgrades \\
		$T_i$ & Time between upgrade $i-1$ and $i$, with $T_0=0$ and $\sum_{i=1}^{N+1} T_i=H$ \\
		$C(t)$ & Cycle costs as specified in~\eqref{eq:CycleCosts} \\
		$S^{\Pi}$ & The number of upgrades not jointly executed with an overhaul under policy~$\Pi$\\
		$\mathcal{K}^{\Pi}(H)$ & The total costs over a lifetime~$H$ under policy~$\Pi$, as given in~\eqref{eq:TotalCostsGeneral}\\
		$\mathcal{C}^{\Pi}(H)$ & The total costs over a lifetime~$H$ under policy~$\Pi$ in the base case ($c_d=0$), as given in~\eqref{eq:TotalCostsBaseCase}\\
		$\bar{N}$ & Upper bound for the number of upgrades under the optimal policy, given by $(C(H)+v(0))/(c_0-v(0))$\\
		$\mathcal{C}^N(H) $ & The total costs if $N$ upgrades are executed after every $H/(N+1)$ time, as given in~\eqref{eq:TotalCostsBaseNUpgrades}\\
		$\mathcal{T}^N$ & Set of candidate solutions for the optimal inter-upgrade times with $N$ upgrades, see~\eqref{eq:CandidateSolutionsNUpgrades}\\
		$\tilde{\mathcal{C}}^N(H,t)$ & Total costs if we execute $N$ upgrades every $(H-t)/N$ time, see~\eqref{eq:TotalCostsBaseSShape}\\
		$\mathcal{K}^{\Pi}(t)$ & The total costs in $(H-t,H]$ under policy~$\Pi$\\
		\hline
	\end{tabular}
\end{table}

\newpage

\section{Pseudo-code for solution approaches}\label{app:SolutionApproach}
\subsection{Solution algorithms for the base case}
\label{sec:SolutionApproachB1}
{We first present Algorithm~\ref{alg:FindOptimalAmongEquidistantConvex} to find the optimal number of upgrades in the base case with a convex cycle cost function. In practice, the optimal number
	of upgrades is  typically not too large, e.g., not exceeding 10 because the time between
	upgrades is often no less than multiple years for an asset with a lifespan of 30 years. So,
	Algorithm~\ref{alg:FindOptimalAmongEquidistantConvex} terminates quickly for realistic instances. Note that a bisection approach could also be used to improve the efficiency of the algorithm.}

\begin{algorithm}[htb]
	\caption{Algorithm to find optimal upgrade policy in base case with convex cycle cost function}
	\label{alg:FindOptimalAmongEquidistantConvex}
	\fbox{\begin{minipage}[b]{0.9\textwidth}
			\begin{algorithmic}[1]
				\REQUIRE Values $H$, $c_0$, and functions $v(t)$, $c_f(t)$, $k(t)$ and $h(t)$ for $t\in [0,H]$.
				\STATE Set $N=0$;\\ 
				\While{$\mathcal{C}^{N+1}(H) < \mathcal{C}^{N}(H)$}{Set $N=N+1$;}
				\RETURN $N^*=N$ and $T_i^*=H/(N+1)$, $i=1,\ldots,N^*+1$. 
			\end{algorithmic}
	\end{minipage}}\\
\end{algorithm}

Algorithm~\ref{alg:FindOptimalBaseSShaped} finds the optimal number of upgrades in the base case with an S-shaped cycle cost function.
\begin{algorithm}[h!]
	\caption{Algorithm to find optimal upgrade policy and corresponding minimum costs in base case with S-shaped cycle costs}
	\label{alg:FindOptimalBaseSShaped}
	\fbox{\begin{minipage}[b]{0.9\textwidth}
			\begin{algorithmic}[1]
				\REQUIRE Values $H$, $c_0$, and functions $v(t)$, $c_f(t)$, $k(t)$ and $h(t)$ for $t\in [0,H]$ (with point of inflection $x$).	
				\STATE Set $N=\lceil H/x -1\rceil$;\\ 
				\While{$\mathcal{C}^{N+1}(H) < \mathcal{C}^{N}(H)$}{Set $N=N+1$;}
				\STATE Set $N^*=N$, $T^*=H/(N+1)$ and $C=\mathcal{C}^{N}(H)$. 
				\IF{$C(H) < C$}
				\STATE Set $C=C(H)$, $N^*=0$ and $T^*=0$;
				\ENDIF
				\FOR{$1 \leq N \leq \bar{N} = \lfloor (C(H)+v(0))/(c_0-v(0)) \rfloor$}
				\STATE Determine $t^* \in[\max\{H/(N+1),x\},H]$ that minimizes $\tilde{\mathcal{C}}^N(H,t^*)$ as in~\eqref{eq:TotalCostsBaseSShape};
				\IF{$t^* \in (x,H)$ and $\tilde{\mathcal{C}}^N(H,t^*) < C$}
				\STATE Set $C=\tilde{\mathcal{C}}^N(H,t^*)$, $N^*=N$ and $T^*=(H-t^*)/N^*$;
				\ENDIF
				\ENDFOR
				\RETURN $N^*$ and $T_i^*=T^*$, $i=1,\ldots,N^*$, $T_{N^*+1}=H-N^* T^*$ and $\mathcal{C}^*(H)=C$. 
			\end{algorithmic}
	\end{minipage}}
\end{algorithm}

Finally, we consider generalized cycle cost functions that can be partitioned as described in Section~\ref{sec:GeneralCycleCosts} with $k$ intervals, and describe how we can (greedily) derive the optimal policy by exploiting several derived properties in our analysis. Suppose that the optimal upgrade policy has $N_i^*$ inter-upgrade times in interval $\mathcal{H}_i$, $i=1,\ldots,k$. Proposition~\ref{prop: general_cycle_cost} and Lemma~\ref{lem:FiniteNumberOfReplacements} show that the number of possibilities for $(N_1^*,\ldots,N_k^*)$ is finite.
%
{More specifically, Proposition~\ref{prop: general_cycle_cost} and Lemma~\ref{lem:FiniteNumberOfReplacements} show that there exists an optimal upgrade policy where the number of upgrades falls in the set
	\begin{align}
		\mathcal{N} := \left\{(n_1,\ldots,n_k) \in \mathbb{N}_{\geq 0}^k : n_i \in \{0,1\} \textrm{ for } i \textrm{ odd}, \;\;\; n_i \leq \bar{N} = (C(H)+v(0))/(c_0-v(0)) \;\; \forall i=1,\ldots,k \right\}.
	\end{align}
	Moreover, write $\underline{A}:=\min_{t\in[0,H]}\{C'(t)\}\geq 0$ and $\bar{A}:=\max_{t\in[0,H]}\{C'(t)\} < \infty$. Since the cycle cost function's derivative exists on $[0,H]$ and is non-negative, we can define for any $\alpha \in [\underline{A},\bar{A}]$,
	\begin{align}
		\mathcal{T}_\alpha := \left\{(t_1,\ldots,t_k) \in \{[0,H]\cup \emptyset\}^k: \left(\textrm{either } t_i \in \mathcal{H}_i  \textrm{ and } C'(t_i) = \alpha \textrm{, or } t_i=\emptyset \right) \;\; \forall i=1,\ldots,k \right\}.
	\end{align}
	In view of~\eqref{eq:NecessaryConditionC} and Proposition~\ref{prop: general_cycle_cost}, we observe that there exists an optimal upgrade policy where the derivative of the inter-upgrade times is given by some $\alpha \in [\underline{A},\bar{A}]$, and hence is contained in the set $\mathcal{T}_{\alpha^*}$.
}

{Consequently, a greedy algorithm to find an optimal upgrade policy can be described as follows: for~every $(n_1,\ldots,n_k) \in \mathcal{N}$, determine all upgrade policies $(t_1,\ldots,t_k) \in \bigcup_{\alpha\in [\underline{A},\bar{A}]} \mathcal{T}_\alpha$ that yield $\sum_{i=1}^k n_i t_i I_{\{t_i\neq \emptyset\}}=H$ (if any). Minimize the total costs under the determined policies, and save the corresponding outcome (i.e.~minimized costs and corresponding inter-upgrade times). Enumerate over all (finite) values of $\mathcal{N}$ to establish the overall minimal costs and corresponding upgrade policy.
%
%
	This greedy algorithm can be used to efficiently obtain the optimal upgrade policy as long as $k$ is not too large.}

\subsection{Solution algorithms for positive penalty costs}
\label{sec:SolutionApproachB2}
{For Algorithm~\ref{alg:FindOptimalNoIntermediate} and Algorithm~\ref{alg:FindOptimalGeneral}, we extend the notation for the total costs $\mathcal{K}^\Pi(t)$ as the total costs in $(H-t,H]$ under policy~$\Pi$. Let $\mathcal{K}^*(t)$ denote the minimal costs over all upgrade policies on this interval. Moreover, we denote $\hat{K}^*(t)$ as the minimal total costs in $(H-t,H]$ among the policies where upgrades are only executed during overhaul moments.}
\begin{algorithm}[htb]
	\caption{Algorithm to find optimal upgrade policy among the class where upgrades are only executed during overhaul moments.}
	\label{alg:FindOptimalNoIntermediate}
	\fbox{\begin{minipage}[b]{0.9\textwidth}
			\begin{algorithmic}[1]
				\REQUIRE Values $H$, $c_0$, and functions $v(t)$, $c_f(t)$, $k(t)$ and $h(t)$ for $t\in [0,H]$.
				\STATE Initiate $\hat{K}^*(M_{m+1})=C(M_{m+1})$ and save $\hat{\mathcal{T}}(m+1) = \{M_{m+1}\} $;
				\FOR{$l=m,\ldots,1$}
				\STATE Set $\hat{K}^*\left(\sum_{i=l}^{m+1} M_{i}\right)=C\left(\sum_{i=l}^{m+1} M_{i}\right)$ and $\hat{\mathcal{T}}(l) = \{\sum_{i=l}^{m+1} M_{i}\}$;		
				\FOR{$j=l,\ldots,m$}
				\STATE Set $x= C\left(\sum_{i=l}^{j} M_{i}\right) + c_0 + \hat{K}^*\left(\sum_{i=j+1}^{m+1} M_{i}\right)$;
				\IF{$x <\hat{K}^*\left(\sum_{i=l}^{m+1} M_{i}\right)$}
				\STATE Set $\hat{K}^*\left(\sum_{i=l}^{m+1} M_{i}\right) = x$ and $\hat{\mathcal{T}}(l) = \{\sum_{i=l}^{j} M_{i}\}\cup \hat{\mathcal{T}}(j+1)$;	
				\ENDIF
				\ENDFOR		
				\ENDFOR
				\RETURN Total costs $\hat{K}^*(H)=\hat{K}^*\left(\sum_{i=1}^{m+1} M_{i}\right)$ and upgrade policy $\hat{\mathcal{T}}(1)$.
			\end{algorithmic}
	\end{minipage}}
\end{algorithm}
%
%
\begin{algorithm}[h!]
	\caption{Algorithm to find optimal upgrade policy.}
	\label{alg:FindOptimalGeneral}
	\fbox{\begin{minipage}[b]{0.9\textwidth}
			\begin{algorithmic}[1]
				\REQUIRE Values $H$, $c_0$, $c_d$ and functions $v(t)$, $c_f(t)$, $k(t)$ and $h(t)$ for $t\in [0,H]$.
				\STATE Initiate ${K}^*(M_{m+1})=\hat{\mathcal{C}}^*(M_{m+1})$ and save corresponding upgrade times ${\mathcal{T}}(m+1)$;
				\FOR{$l=m,\ldots,1$}
				\STATE Set ${K}^*\left(\sum_{i=l}^{m+1} M_{i}\right)=\hat{\mathcal{C}}^*\left(\sum_{i=l}^{m+1} M_{i}\right)$ and save corresponding upgrade times ${\mathcal{T}}(l)$;		
				\FOR{$j=l,\ldots,m$}
				\STATE Set $x= \hat{\mathcal{C}}^*\left(\sum_{i=l}^{j} M_{i}\right) + c_0 + {K}^*\left(\sum_{i=j+1}^{m+1} M_{i}\right)$ and save upgrade times $\mathcal{T}$ on $\left[\sum_{i=1}^{l-1} M_{i}, \sum_{i=1}^{j} M_{i} \right]$;
				\IF{$x < {K}^*\left(\sum_{i=l}^{m+1} M_{i}\right)$}
				\STATE Set ${K}^*\left(\sum_{i=l}^{m+1} M_{i}\right) = x$ and ${\mathcal{T}}(l) = \mathcal{T} \cup {\mathcal{T}}(j+1)$;	
				\ENDIF
				\ENDFOR		
				\ENDFOR
				\RETURN Total costs ${K}^*(H)={K}^*\left(\sum_{i=1}^{m+1} M_{i}\right)$ and upgrade policy ${\mathcal{T}}(1)$.
			\end{algorithmic}
	\end{minipage}}
\end{algorithm}
{We point that the only difference between Algorithm~\ref{alg:FindOptimalNoIntermediate} and Algorithm~\ref{alg:FindOptimalGeneral} is that in step 1, 3 and 5, the cycle cost function $C(\cdot)$ is replaced by $\hat{\mathcal{C}}(\cdot)$. In other words, we need to find the optimal upgrade policy in the base case with upgrade price $c_0+c_d$. Section~\ref{sec:SolutionApproachB2}.1 provides the algorithms to do so in case of convex, S-shaped and generalized cycle costs.
}

\newpage
\section{Proofs}\label{app:proofs}
\begin{proof}[Proof of Lemma~\ref{lem:FiniteNumberOfReplacements}]
Write $\Pi_0$ as the policy to never upgrade, and note that $\mathcal{K}^{\Pi_0}(H) = C(H)$. For any policy~$\Pi$ with $N > \bar{N}$, $\mathcal{K}^\Pi(H) \geq N c_0+(N+1)(-v(0)) > \bar{N} (c_0-v(0)) -v(0) = C(H) = \mathcal{K}^{\Pi_0}(H) \geq \mathcal{K}^*(H).$
\end{proof}

\subsection{Proofs for the base case}\label{app:proofsBaseCase}

\begin{proof}[Proof of Proposition~\ref{prop:EquidistantResult}]
We need to minimize the total costs by finding the optimal policy that describes the number of upgrades $N^*$ and corresponding $T_1^*,\ldots,T^*_N$ such that
\begin{align}
	\mathcal{C}^*(H) = \min_\Pi \mathcal{C}^\Pi(H).
\end{align}
Due to Lemma~\ref{lem:FiniteNumberOfReplacements}, we know that it is never optimal to upgrade more than $\bar{N}$ times. First, we answer the question what the optimal inter-upgrade times are if the optimal number of upgrades is given by $N^*$. 

If $N^*=0$, then there are no upgrades during the lifetime and $\mathcal{C}^*(H) = C(H)$. Otherwise, $1 \leq N^* \leq \bar{N}$ and the policy also needs to describe the inter-upgrade times. Then,
\begin{align*}
	\mathcal{C}^*(H) = N^* c_0 + \min_{\substack{T_1,\ldots,T_{N^*+1}  \geq 0,\\ T_1+\ldots+T_{N^*+1} = H}} \left\{ \sum_{i=1}^{N^*+1} C(T_i)   \right\}.
\end{align*}
The term in the minimization operation is the summation of $N^*+1$ identical convex non-decreasing functions in $T_i$, under the condition that the sum of the inter-upgrade times $T_i$ is equal to $H$. By the convexity property, this is minimized if $ T_1=\ldots=T_{N^*+1}=H/(N^*+1)$. Using this observation together with Lemma~\ref{lem:FiniteNumberOfReplacements} concludes the proof.
\end{proof}

\begin{proof}[Proof of Lemma~\ref{lem:ConvexityOfCN}]
To show that $\mathcal{C}^N(H)$ is convex in $N$, we require that $\Delta^2 \mathcal{C}^N(H)  \geq 0$. Note that
\begin{align*}
	\Delta \mathcal{C}^N(H) = \mathcal{C}^{N+1}(H) - \mathcal{C}^N(H) = c_0 + C\left(\frac{H}{N+2}\right) - (N+1) \left(C\left(\frac{H}{N+1}\right)-C\left(\frac{H}{N+2}\right)\right),
\end{align*}
and hence 
\begin{align*}
	\Delta^2 \mathcal{C}^N(H) &= \Delta \mathcal{C}^{N+1}(H) - \Delta \mathcal{C}^N(H) \\
	&= (N+1) \left(C\left(\frac{H}{N+1}\right)-C\left(\frac{H}{N+2}\right)\right) - (N+3) \left(C\left(\frac{H}{N+2}\right)-C\left(\frac{H}{N+3}\right)\right).
\end{align*}
Write 
\begin{align*}
	y_i = \frac{H}{N+2} + i \frac{H}{(N+1)(N+2)(N+3)}, \hspace{2cm} z_i = \frac{H}{N+3} + i \frac{H}{(N+1)(N+2)(N+3)},
\end{align*}
and note that $y_{N+3}=H/(N+1)$ and $z_{N+1}=y_0 = H/(N+2)$. Therefore,
\begin{align*}
	\Delta^2 \mathcal{C}^N(H) &= (N+1) \sum_{i=1}^{N+3} \left(C(y_{i})-C(y_{i-1})\right) - (N+3)  \sum_{j=1}^{N+1} \left(C(z_{j})-C(z_{j-1})\right) \\
	&= \sum_{i=1}^{N+3} \sum_{j=1}^{N+1} \left(\left(C(y_{i})-C(y_{i-1})\right) - \left(C(z_{j})-C(z_{j-1})\right)\right)\geq 0,
\end{align*}
where the final inequality follows since every single term within the summations is non-negative due to the convexity and the non-decreasing property of the cycle cost function.
\end{proof}

\begin{proof}[Proof of Lemma~\ref{lem:ConcaveNeverUpgrade}]
Suppose that there is an optimal policy~$\Pi$ for which it is optimal to upgrade exactly $N \geq 1$ times. Since $C(t)$ is concave and non-decreasing in $t\geq 0$, the term $\sum_{i=1}^{N+1} C(T_i)$ is minimized if $T_1=\ldots=T_N=0$ and $T_{N+1}=H$. However, in that case it follows that
\begin{align}
	\mathcal{C}^{\Pi}(H) = N (c_0 - v(0) ) + C(H) > C(H) , 
\end{align}
where the latter corresponds to the total costs under the policy to never execute an upgrade. This contradicts the hypothesis.
\end{proof}

{
	\begin{lemma}\label{lem:SShapedPropertyUpgradeTime}
		Let the cycle costs have an S-shape with point of inflection $x \in (0,H)$ that satisfies the technical requirement. If $N^*\geq 1$, then there exists an optimal upgrade policy where for all $i=1,\ldots,N^*+1$ it holds that $T_i^* \not\in \left(\min\{x,H/(N^*+1)\},\max\{x,H/(N^*+1)\} \right)$.
	\end{lemma}
}

\begin{proof}[Proof of Lemma~\ref{lem:SShapedPropertyUpgradeTime}]
This lemma only applies to cases where $x \neq H/(N+1)$, otherwise\\ $\left(\min\{x,H/(N^*+1)\},\max\{x,H/(N^*+1)\} \right) =\emptyset$. Therefore, we assume $x \neq H/(N+1)$ for the remainder of the proof. Suppose that there would be an optimal strategy $\Pi$ with an inter-upgrade time strictly between point of inflection $\min\{x,H/(N+1)\}$ and $\max\{x,H/(N+1)\}$ with $N\geq1$. Due to symmetry, we can assume without loss of generality that 
\begin{align*}
	T_{N+1} \in \left(\min\left\{x,\frac{H}{N+1}\right\},\max\left\{x,\frac{H}{N+1}\right\}\right).
\end{align*} 
Moreover, note that a necessary condition for an policy to be optimal is 
\begin{align}
	C'(T_i)= C'(T_{N+1}) \hspace{3cm} \forall i=1,\ldots,N.
	\label{eq:NecessaryConditionV}
\end{align}

Moreover, let 
\begin{align}
	&a = \inf\{t \in [0,x] : C'(t)=C'(T_{N+1})\},  \hspace{0.5cm} & c = \inf\{t \in [x,H] : C'(t)=C'(T_{N+1})\},\\
	&b = \sup\{t \in [0,x] : C'(t)=C'(T_{N+1})\}, \hspace{0.5cm} & d = \sup\{t \in [x,H] : C'(t)=C'(T_{N+1})\}.
\end{align}
We point out that at least $a$ and $b$ or $c$ and $d$ are well-defined (or all). Moreover, if it is well-defined, note that $a =b$ in case of strict convexity on $[0,x]$ and $c=d$ in case of strict concavity on $[x,H]$. Since $C(\cdot)$ has an S-shape with point of inflection $x$ that satisfies the technical requirement, we observe that $[a,b]$ and $[c,d]$ contain all solutions of~\eqref{eq:NecessaryConditionV}, and $C(\cdot)$ is linear on both intervals (if they exist). Therefore, without loss of generality, we can say that $T_i \in [a,b]$ for all $i=1,\ldots,N_1$ and $T_i \in [c,d]$ for all $i=N_1+1,\ldots,N_1+N_2$ where $N_1+N_2=N+1$. Note that $N_1=0$ if $a,b$ do not exist, and similarly, $N_2=0$ if $c,d$ do not exist.

\textbf{Case $x < H/(N+1)$:} or in other words, $x < T_{N+1} < H/(N+1)$. In that case $c,d$ are well-defined values and $T_{N+1} \in [c,d]$ and $N_2\geq 1$. If $N_2=1$, then $N=N_1 \geq 1$ and $a,b$ must also be well-defined. Therefore, $T_i \leq b \leq x < H/(N+1)$ for all $i=1,\ldots,N$. Recalling that $T_{N+1} < H/(N+1)$ yields
\begin{align*}
	\sum_{i=1}^{N+1} T_i < \sum_{i=1}^{N+1} H/(N+1)  = H,
\end{align*}
which is a contradiction. Therefore, we require that $N_2 \geq 2$. In that case, we can construct the policy $\tilde{\Pi}$ where
\begin{align*}
	\tilde{T}_i=T_i, \;\;\; i=1,\ldots,N-1, \hspace{1cm} \tilde{T}_{N} = T_N+T_{N+1}-x, \hspace{1cm} \tilde{T}_{N+1} = x.
\end{align*}
Note that $\tilde{T}_{N} \in (x,H)$. Since the cycle costs are concave and due to the technical requirement, $C(T_N) + C(T_{N+1}) > C(\tilde{T}_N) + C(\tilde{T}_{N+1})$, and hence
\begin{align}
	\mathcal{C}^{\tilde{\Pi}}(H) = \mathcal{C}^{\Pi}(H) + C(\tilde{T}_N) + C(\tilde{T}_{N+1}) - C(T_N) -C(T_{N+1}) < \mathcal{C}^{\Pi}(H).
\end{align}
In other words, our original policy $\Pi$ is not optimal, contradicting our hypothesis.

\textbf{Case $x > H/(N+1)$:} or in other words $H/(N+1) < T_{N+1} < x$. In that case, $a$ and $b$ are well-defined and $T_{N+1} \in [a,b]$ and $N_1 \geq 1$.  We will show that if this policy $\Pi$ exists, then there is another policy $\tilde{\Pi}$ that yields the same minimal total costs but has no inter-upgrade time in $(H/(N+1),x)$. First, if $T_i \in [a,b]$ for all $i=1,\ldots,N$, then let policy $\tilde{\Pi}$ be given by $\tilde{T}_i=H/(N+1)$ for all $i=1,\ldots,N+1$. By construction, it holds that $\mathcal{C}^{\tilde{\Pi}}(H) = \mathcal{C}^{\Pi}(H)$ since the cycle costs are linear on $[a,b]$, and hence $\tilde{\Pi}$ is also an optimal upgrade policy. 

Next, suppose that $T_i \not\in [a,b]$ for all $i=1,\ldots,N$, implying directly that $c$ and $d$ are well-defined. Without loss of generality, we can order the inter-upgrade times such that $T_i \in [c,d]$ for $i=1,\ldots,l$ for some $l \geq 1$ and $T_i \in [a,b]$ for $i=l+1,\ldots,N$ (if any). Consider the policy
\begin{align*}
	\tilde{T}_i=T_i, \;\;\; i=1,\ldots,l, \hspace{1cm} \tilde{T}_{N+1} = \tilde{T}_j = \frac{\sum_{i=l+1}^{N+1} T_i}{N-l+1}, \;\;\; j=l+1,\ldots,N.
\end{align*}
Note that $\tilde{T}_{N+1} \geq a$ and 
\begin{align*}
	\tilde{T}_{N+1} = \frac{\sum_{i=l+1}^{N+1} T_i}{N-l+1} = \frac{H-\sum_{i=1}^l T_i}{N-l+1} < \frac{H- l H/(N+1)}{N-l+1}  = \frac{H}{N+1} <x.
\end{align*}
That is, policy $\tilde{\Pi}$ has no inter-upgrade time in $(H/(N+1),x)$, and  $\mathcal{C}^{\tilde{\Pi}}(H) = \mathcal{C}^{\Pi}(H)$ since the cycle costs are linear on $[a,b]$.
\end{proof}

{
	\begin{lemma}
		Let the cycle costs have an S-shape with point of inflection $x \in (0,H)$, where $x$ satisfies the technical requirement. If $N^* \geq 1$, then there exists an optimal policy for which $T_i^* \leq x$ for some $i=1,\ldots,N^*+1$.
		\label{lem:NotOnlyUpgradeAfterTurningPoint}
	\end{lemma}
}

\begin{proof}[Proof of Lemma~\ref{lem:NotOnlyUpgradeAfterTurningPoint}]
Suppose that the statement is not true, and $T_1,\ldots,T_{N+1} > x$. In view of Lemma~\ref{lem:SShapedPropertyUpgradeTime}, we observe that 
\begin{align*}
	\sum_{i=1}^{N+1} T_i \geq \sum_{i=1}^{N+1} \frac{H}{N+1} = H,
\end{align*}
where $\sum_{i=1}^{N+1} T_i=H$ holds if and only if $T_1=\ldots=T_{N+1}=H/(N+1) >x$. In other words, if $T_1,\ldots,T_{N+1} > x$, then $T_1=\ldots=T_{N+1}=H/(N+1) >x$, and the total costs under this policy are given by $N c_0 + (N+1) C(H/(N+1))$.

Consider the policy $\tilde{\Pi}$ where $\tilde{T}_1=\ldots=\tilde{T}_N=x$ and $\tilde{T}_{N+1}=H-N x$. Then, due to the technical requirement and the notion that we have a setting where~$x<H/(N+1)$,
\begin{align*}
	\mathcal{C}^\Pi (H) = N c_0 + N C(x) + C(H-n x) <  N c_0 + (N+1) C(H/(N+1)),
\end{align*}
contradicting the hypothesis.
\end{proof}

{
	\begin{proof}[Proof of Proposition~\ref{prop:OptimalUpgradePolicyBase}]
	If $N^*=0$, then the proposition holds and there is nothing to prove. Therefore, let $N^*\geq1$ for the remainder of the proof. That is, suppose there is an optimal policy where we upgrade $N^*=N_1+N_2-1 \geq 1$ times, where $N_1$ and $N_2$ are the number of times we have inter-upgrade times in the interval $[0,x]$ and $(x,H]$, respectively. In view of~\eqref{eq:NecessaryConditionC}, we observe that there exists such an optimal policy that satisfies $T_1=\ldots=T_{N_1} \leq x$ and $T_{N_1+1}=\ldots=T_{N_1+N_2} >x$ (if $N_2 >0$). We point out that $N_1 \geq 1$ due to Lemma~\ref{lem:NotOnlyUpgradeAfterTurningPoint}. To conclude the result, what remains to show is that $N_2 \in \{0,1\}$. To that purpose, suppose that $N_2\geq2$. Consider the policy $\tilde{\Pi}$ with 
	\begin{align*}
		\tilde{T}_i=T_1, \;\;\; i=1,\ldots,N_1,& \hspace{1cm} \tilde{T}_j=T_{N^*+1}, \;\;\; j=N_1+1,\ldots,N^*-1, \\
		\tilde{T}_{N^*} = x,& \hspace{1cm} \tilde{T}_{N^*+1} = 2 T_{N^*+1}-x.
	\end{align*}
	Note that this solution is feasible, since
	\begin{align*}
		\sum_{i=1}^{N_1+N_2} \tilde{T}_i = N_1 T_1 + (N_2-2) T_{N^*+1} + x + (2 T_{N^*+1}-x) =  N_1 T_1 + N_2 T_{N^*+1} = H.
	\end{align*}
	Due to the technical requirement, we observe that
	\begin{align*}
		\mathcal{C}^{\tilde{\Pi}} (H) &= N^* c_0 +N_1 C(T_1) + (N_2-2) C(T_{N^*+1}) + C(x) + C(2 T_{N^*+1}-x) \\
		&< N^* c_0 +N_1 C(T_1) + N_2 C(T_{N^*+1}) =  \mathcal{C}^{\Pi} (H).
	\end{align*}
	This contradicts the hypothesis, proving our proposition.
	\end{proof}
}

{
	\begin{proof}[Proof of Proposition~\ref{prop: general_cycle_cost}]
	The result directly follows from the same arguments used in the proof of Proposition~\ref{prop:OptimalUpgradePolicyBase}, so a detailed proof is omitted.
	\end{proof}
}

\subsection{Proofs for the sensitivity analysis}\label{app:sensitivityProofs}

\begin{proof}[Proof of Proposition~\ref{prop:MainCdResult}]
This is a direct consequence of the total cost function~\eqref{eq:TotalCostsGeneral}. Note that for any fixed policy~$\Pi$, the total cost function is linear in $c_d$ with slope $S^{\Pi}$. Therefore, as penalty $c_d$ increases, the optimal policy can only change to one with a smaller value of~$S^\Pi$.
\end{proof}

\begin{proof}[Proof of Lemma~\ref{lem:NDecreasingInc0}]
For every policy~$\Pi$, the total costs are given by~\eqref{eq:TotalCostsGeneral}. Write $\mathcal{P}_N$ as the set of all possible upgrade policies with $N$ upgrades. We observe that for every $N \in \mathbb{N}$, we can determine the upgrade times that minimize the total costs by solving
\begin{align}
	K^N(H) := \min_{\Pi \in \mathcal{P}_N} S^\Pi c_d + \sum_{i=1}^{N+1} C(T_i),
\end{align}
regardless of the value of~$c_0$. In particular,
\begin{align*}
	\mathcal{K}^*(H) = \min_{N \in \mathbb{N}} \left\{ N c_0 + K^N(H)\right\}.
\end{align*}
We point out that
\begin{align*}
	\frac{\partial}{\partial c_0} \left(N c_0 + K^N(H) \right) = N.
\end{align*} 
That is, it is linear in $c_0$ with slope $N$, and $N^*$ is the argument that minimizes the total costs. In conclusion, the optimal number of upgrades will never decrease.
\end{proof}

\begin{proof}[Proof of Lemma~\ref{lem:SteepIncreaseCycleCostsAfterZ}]
For any general variable $X$, let $\tilde{X}$ denote the corresponding one in the alternate setting. For any policy~$\Pi$, it holds that $\mathcal{K}^\Pi (H) \geq \mathcal{\tilde{K}}^\Pi(H)$ since $\tilde{C}(t) \leq C(t)$ for all $t \in [0,H]$. Moreover, since $\tilde{T}_i^* \leq z$ for all $i=1,\ldots,\tilde{N}^*+1$, it holds that $\mathcal{K}^{\tilde{\Pi}^*} (H) =  \mathcal{\tilde{K}}^*(H)$. In conclusion, we obtain
\begin{align*}
	\mathcal{K}^{\tilde{\Pi}^*} (H) =  \mathcal{\tilde{K}}^*(H) \leq \mathcal{\tilde{K}}^\Pi(H) \leq \mathcal{K}^\Pi (H),
\end{align*} 
for any policy~$\Pi$. We can conclude that $\tilde{\Pi}^*$ is also optimal for the original setting.
\end{proof}

\end{document}